\documentclass[12pt,reqno]{amsart}
\usepackage{amsfonts}
\usepackage{amsmath,amssymb,amsthm, mathrsfs,amsopn}
\usepackage{a4wide}
\allowdisplaybreaks
\numberwithin{equation}{section}
 \usepackage[pagewise]{lineno}\linenumbers\nolinenumbers
\usepackage{graphicx}
\usepackage{epstopdf}
\usepackage{float}
\usepackage{stfloats}
\usepackage{subfigure}
\usepackage{soul}
\usepackage{url}
\usepackage{color}
\usepackage[colorlinks, linkcolor=blue, citecolor=blue]{hyperref}

\newtheorem{theorem}{Theorem}[section]

\newtheorem{lemma}[theorem]{Lemma}

\theoremstyle{definition}

\newtheorem{remark}[theorem]{Remark}

\newcommand{\R}{\mathbb{R}}

\newcommand{\eeq}{\end{equation}}
\newcommand{\beq}{\begin{equation*}}

\begin{document}
	
	\title[prescribed the mass solution]{Existence and blow up behavior of  prescribed mass solutions on large smooth domains to the Kirchhoff equation with  combined nonlinearities}
	\thanks{The research was supported by National Science Foundation of China *****}
\author{Xiaolu Lin}
	\address{ Xiaolu Lin,~School of Mathematics and Statistics, Central China Normal University, Wuhan 430079, P. R. China}
	\email{ xllin@ccnu.edu.cn}

	\author{Zongyan Lv}
	\address{ Zongyan Lv,~ Center for Mathematical Sciences, Wuhan University of Technology, Wuhan 430070, P. R. China	}
	\email{zongyanlv0535@163.com}
	\date{}

\begin{abstract}
In this paper,  we consider the existence, multiplicity and the asymptotic behavior  of prescribed mass
solutions to the following nonlinear Kirchhoff equation with mixed nonlinearities:
\begin{equation*}
    \begin{cases}
        -(a+b\int|\nabla u|^2\mathrm{d}x)\Delta u+V(x)u+\lambda u=|u|^{q-2}u+\beta|u|^{p-2}u \quad&\text{in}\ \Omega,\\
        \int_{\Omega}|u|^2\mathrm{d}x=\alpha,
    \end{cases}
\end{equation*}
both on large bounded smooth star-shaped domain $\Omega\subset\mathbb{R}^{3}$ and on $\mathbb{R}^{3}$, where  $2<p<\frac{14}{3}<q<6$ and $V(x)$ is the potential. The standard approach based on the Pohozaev identity to obtain normalized solutions is invalid due to the presence of potential $V(x)$.

		\vskip 0.2cm
		\noindent{\bf MR(2010) Subject Classification:} {35J60; 35J20; 35R25}
	\end{abstract}

\maketitle

\section{Introduction}

	The aim of this paper is to  study the existence, multiplicity and the asymptotic behavior of the  solutions to the following nonlinear Kirchhoff equation:
		\begin{equation}\label{main-eq}
			\begin{cases}
				-(a+b\int_\Omega|\nabla u|^2\mathrm{d}x)\Delta u+V(x)u+\lambda u=|u|^{q-2}u+\beta|u|^{p-2}u\quad&\text{in}\ \Omega,\\
				u\in H_0^1(\Omega),\quad\int_{\Omega}|u|^2dx=\alpha,
			\end{cases}
		\end{equation}
	where $a,b>0$, $\Omega \subset \mathbb{R}^3$ is either a bounded smooth star-shaped domain or all of $\mathbb{R}^3$,  $2<p<\frac{14}{3}<q<6$ and $V(x)$ is the potential. The mass $\alpha>0$ and the parameter $\beta \in \mathbb{R}$ are prescribed. Here  the frequency $\lambda$  appears as a Lagrange multiplier.

As is known to us all, the equation of \eqref{main-eq} originates from
the following   problem of Kirchhoff type:
   \begin{equation}\label{problem1}
-\left(a+b\int_{\R^3} |\nabla u|^2 \mathrm{d}x  \right)\Delta u +V(x)u+\lambda u=f(u)   ~~\text{in}~~\R^3,
\end{equation}
where $f \in \mathcal{C}(\mathbb{R}, \mathbb{R})$, which was proposed by G. Kirchhoff in \cite{Ki} as an extension of the classical D'Alembert's wave equations, describing free vibrations of elastic strings.
From a mathematical point of view,   the Kirchhoff equation  is no longer a pointwise identity   due to the appearance of the term $(\int_{\R^3} |\nabla u|^2dx)\Delta u$.   After the a functional analysis approach was introduced by J.L. Lions in the pioneer work \cite{JLLions1978} the Kirchhoff type equations began to call attention of  many researchers.
	In the equation \eqref{main-eq}, if $\lambda $ is given, we call it fixed frequency problem.
For a fixed frequency $\lambda \in \R  $,	the existence and multiplicity of solutions to \eqref{main-eq} has been investigated in the last two decades by many authors, e.g. \cite{F,G,LY,LZ24,WT} for a survey on almost classical results.

Nowadays, physicists are more interested in solutions  possessing  the prescribed $L^{2}$-mass, which is usually  called normalized solution.
The existence and properties of these normalized solutions recently has attracted the attention of many researchers. The literature in this direction is huge and we do not even make an attempt to summarize it here (\cite{BV13,BQZ24,GJ,BHG3,PPV}).
The normalized solutions of \eqref{problem1} can be searched as critical points of $J_V(u)$ constrained on $S_\alpha$, where
\begin{equation}\label{sec-1-J(u)}
J_V(u)=\frac{a}{2}  \int_{\R^3}|\nabla u|^2dx+\frac{b}{4}\bigg(\int_{\R^3}|\nabla u|^2dx\bigg) ^2+\frac{1}{2} \int_{\R^3} V(x)u^2\mathrm{d}x-\int_{\R^3}F(u)\mathrm{d}x
\end{equation}
and
 $S_\alpha=\{u \in H^1(\mathbb{R}^{3}):\int_{\R^3}|u|^2dx=\alpha \}$.
		In this case, the mass $\alpha > 0$  is prescribed, while the frequency $\lambda$ is unknown and  comes out as a Lagrange multiplier.  	
As we know,  Ye seems to be the  first  to consider the  normalized solutions to Kirchhoff equation \eqref{problem1} with $V(x)\equiv 0$ and  $f(x,u)=|u|^{p-2}u$. Based on  the paper of Ye \cite{Y1,Y2}, we can summarize that:

(i) If $2<p<2+\frac{4}{N}$, for $\alpha>0$, there exists a global minimizer. If $2+\frac{4}{N}\leq p<2+\frac{8}{N}$, there exists a $\alpha_p^*\geq 0$  such that $\alpha>\alpha_p^*$,  the energy functional has at least one constraint critical point on $S_\alpha$.

(ii) If $p\geq 2+\frac{8}{N}$, there is no minimizers.  However, there exits  a mountain pass type solution for the energy functional  constrained on $S_\alpha$.

(iii) If $p=2+\frac{8}{N}$ and $\alpha>\alpha^*$,    there exits  a mountain pass type solution for the energy functional  constrained on $S_\alpha$.

As a consequence, we notice that the $L^2$-critical exponent $p=2+\frac{8}{N}$   is the threshold exponent  which plays a pivotal role in the study of normalized solutions to the Kirchhoff problem.

We would like to mention here that some paper involved in dealing with the Kirchhoff equation with the autonomous case, i.e., the potential $V$ is a constant.
In \cite{LLY}, Li, Luo and Yang considered the existence and asymptotic properties  of normalized solutions to the following Kirchhoff equation with  mixed nonlinearities:
 \begin{equation}\label{problem2}
-\left(a+b\int_{\R^3} |\nabla u|^2 \mathrm{d}x \right)\Delta u +\lambda u=|u|^{p-2}u+\mu |u|^{q-2}u   ~~\text{in}~~\R^3,
\end{equation}
 where $a, b, \alpha, \mu>0, 2<q<\frac{14}{3}<p\leq 6$ or $\frac{14}{3}<q<p\leq 6$, and proved a multiplicity result for the case of  $2<q<\frac{10}{3}$
 and $\frac{14}{3} <p<6$, and the existence of ground state normalized solutions for $2<q<\frac{10}{3}<p=6$ or $\frac{14}{3}<q<p\leq 6.$ They also gave some asymptotic results on the obtained  normalized  solutions.  On the one hand,  Chen and Tang obtained a local minimizer and a mountain-pass type solution with explicit conditions on $\alpha$ in the case of $2<q<\frac{10}{3}$. In particular, they fill the gap of $\frac{10}{3}\leq q<\frac{14}{3}$, which a mountain-pass type solution has been found under explicit conditions on $\alpha$. They also succeed to obtain the existence of a ground state solution for all $\alpha>0$ when $\frac{14}{3} \leq q<6$.
 In \cite{FLZ2023}, Feng, Liu and Zhang considered the existence and multiplicity of a normalized solution  to the Kirchhoff equation with combined nonlinearities, which is the Sobolev critical case.
Zeng et al. \cite{ZZZZ2021} showed the existence, nonexistence and multiplicity of the normalized solutions to the Kirchhoff equation with a general nonlinear term based on the scaling skills, and the results about the existence, nonexistence  and multiplicity of the normalized solutions to Sch\"{o}dinger equation in \cite{JeanZhangZhong2021}.

We would like to mention here some results involving the potential
$V(x)$ is non-constant. In \cite{CZ23}, Cai and Zhang first proved the existence of normalized solutions under suitable conditions on the potential $V(x)\ge0$. Secondly, they  proved the existence of mountain pass solutions with positive energy and the nonexistence of solutions with negative energy if $V(x)\le0$ is bounded.  Thirdly, they could find a local minimizer of the energy functional if $V(x)\le0$ is not too small on a suitable interval.
Rong and Li \cite{RL23} studied the existence of positive normalized ground state solutions to the Kirchhoff equation with an explicit assumption on potential $V$. In addition, He, Lv and Tang \cite{HLT} studied the existence of ground state normalized solutions to the Kirchhoff equation with potential and general nonlinear term. In \cite{HLT}, the assumptions on $V(x)$ are defined as
follows:
\begin{itemize}
\item[(V-1)]
 $\lim\limits_{|x|\to+\infty}V(x)=\sup\limits_{x\in \R^3} V(x) =0$ and there exists some $\sigma_1 \in [0, \frac{3(\alpha-2)-4}{3(\alpha-2)}a)$  such that
\begin{equation*}
\left|\int_{\R^3} V(x)u^2\mathrm{d}x\right|\leq\sigma_1\|\nabla u\|_2^2,~\text{for all}~u\in H^1(\R^3).
\end{equation*}

\item[(V-2)] $\nabla V(x)$ exists for a.e. $x\in \R^3$.  Putting  $W(x):=\frac{1}{2}\langle\nabla V(x),x\rangle$,  there exists some $\sigma_2\in[0, \frac{3(\alpha-2)(a-\sigma_1)}{4}-a ]$ such that
\begin{equation*}
\left|\int_{\R^3} W (x)u^2\mathrm{d}x\right|\leq\sigma_2\|\nabla u\|_2^2,~\text{for all}~u\in  H^1(\R^3).
\end{equation*}

\item[(V-3)] $\nabla W(x)$ exists for a.e. $ x\in \R^3$. Letting  $\Upsilon (x):=4W(x)+\langle\nabla W(x),x\rangle,$
    there exists some $\sigma_3 \in [0, 2a) $ such that
\begin{equation*}
\int_{\R^3} \Upsilon_{+} (x) u^2\mathrm{d}x\leq\sigma_3\|\nabla u\|_2^2,~\text{for all}~u \in  H^1(\R^3).
\end{equation*}


\end{itemize}
It is worth to mention that the assumptions on $V$ in \cite{HLT} and \cite{RL23} are inspired by the literature \cite{DZ22}, which are different from our assumptions on potential $V$.
Since  our methods in this paper depends on the monotonicity trick,  no longer depending on the so-call Pohozaev  mainfold which is composed of Pohozaev identy. In other words, our assumptions in this paper is more weak and more generous.

		There are only a few approaches and results on the study of normalized solutions on bounded domains \cite{ABD,BCJ,BCJS10,BHG3,WC2024}, which dealt with the  autonomous case. In view of the inherent characteristics of the problem with prescribed mass, the so-called Pohozaev manifold is not available when working in bounded domains.  However,  it is worth to point that  the method of these papers above do not work for non-constant $V$.
        Recently, Bartsch-Qi-Zou \cite{BQZ24} considered the existence and properties of normalized solutions with a combination of  mass subcritical and  mass supercritical, i.e. $f(|u|)u=|u|^{q-2}+\beta|u|^{p-2}u$, with $2<p<2+\frac{4}{N}<q<2^*$.

Inspired by \cite{BQZ24},  as naturally
expected, we are going to extend
the results of  \cite{BQZ24} to the Kirchhoff equation. Compared with the above problem,  the study of the
convergence of Palais-Smale sequences becomes more complicated as the presence of the nonlocal term  $\int_{\R^3} |\nabla u|^2~dx$.
             By subtle energy estimates, we can obtain the existence of normalized solution to \eqref{main-eq} on large bounded smooth star-shaped domains $\Omega_r$. Next, by studying the asymptotic behavior of the radius $r\to \infty$, we also obtain the existence of normalized solution in whole space. Furthermore, based on the existence results in the whole space, we even study the asymptotic behavior of $\alpha\to 0^+$ via a blow up argument.
 Up to our knowledge,  the present paper seems to be the first result  for normalized solutions to the  Kirchhoff equations with potential  and combined nonlinearities   in $\Omega$ which is large bounded domain even expand to the whole space. Specially, making clear the asymptotic behavior of these normalized solutions for $\alpha\to 0^+$ to Kirchhoff equations  with potential and combined nonlinearities. Last but not least, we consider the bifurcation property in the case of mass subcritical.

		Throughout this paper, we will use the following notation: Let $m_+:=\max\{m,0\}$, $m_-:=\min\{m,0\}$ with $m\in\mathbb{R}$. For $\Omega\subset\mathbb{R}^{3}$, $r>0$ we set
		$
		\Omega_r:=\Big\{rx\in\mathbb{R}^{3}:x\in\Omega\Big\}
		$
		and
		\[
		S_{r,\alpha}:=S_{\alpha}\cap H_0^1(\mathbb{R}^{3})=\Big\{u\in H_0^1(\Omega_r):\|u\|_{L^2(\Omega_r)}^2=\alpha\Big\}.
		\]
		without loss of generality,  we assume that $\Omega\subset\mathbb{R}^{3}$ is a bounded smooth domain with $0\in\Omega$ and star-shaped with respect to 0.  Let $S$ the optimal constant of the Sobolev embedding from $D^{1,2}(\mathbb{R}^{3})$ to $L^{6}(\mathbb{R}^{3})$.
		Before stating our main results, we  state our basic assumptions on the potential,
		\begin{description}
			\item[$(V_0)$] $V\in C(\mathbb{R}^{3})\cap L^{\frac{3}{2}}(\mathbb{R}^{3})$ is bounded and $\|V_-\|_{\frac{3}{2}}<aS$.
		\end{description}
		For some results we will require that $V$ is $C^1$ and define
		$\tilde{V}:\mathbb{R}^{3}\to\mathbb{R}$ by $\tilde{V}(x)=\nabla V(x)\cdot x$.
		In our first result, we consider the case $\beta>0$.

\begin{theorem}\label{beta>0-e<0-Omega}
Suppose that $(V_0)$, $\beta>0$ hold and set $2<p<\frac{10}{3}<\frac{14}{3}<q<6$ and
\[
\alpha_V=\bigg(\frac{b}{12(q-p)}\bigg)^3\bigg(\frac{q(14-3p)}{C_q}\bigg)^\frac{14-3p}{q-p}\bigg(\frac{p(3q-14)}{\beta C_p}\bigg)^\frac{3q-14}{q-p}.
\]
Then the following hold for $0<\alpha<\alpha_V$:
\begin{description}
  \item[(i)]There exists $r_\alpha>0$ such that \eqref{main-eq} on $\Omega_r$ with $r>r_\alpha$ has a local minimum type solution $(\lambda_{r,\alpha},u_{r,\alpha})$ with $u_{r,\alpha}>0$ in $\Omega_r$ and $E_V(u_{r,\alpha})<0$. Moreover, there exists $C_\alpha>0$ such that \[\limsup\limits_{r\to\infty}\max\limits_{x\in\Omega_r}u_{r,\alpha}(x)<C_\alpha.
  \]
  \item[(ii)] Moreover, if there exists $m>0$ independent of $r$ such that
  \begin{equation}\label{vvv}
  \frac{1}{2}\int_{\Omega_r}\tilde{V}u^2\mathrm{d}x-\int_{\Omega_r}Vu^2\mathrm{d}x\ge m>0,
  \end{equation}
  then there holds
  \[
  \underset{r\to\infty}{\liminf}\,\lambda_{r,\alpha}>0.
  \]
\end{description}
\end{theorem}

\begin{theorem}\label{beta>0-e>0-Omega}
Suppose that $(V_0)$ and $\beta>0$ hold,  $V(x)\in C^1$ and $\tilde{V}$ is bounded.
Then the following hold:
\begin{description}
  \item[(i)]There exist $\tilde{\alpha}_V, \tilde{r}_\alpha>0$ such that for $r>\tilde{r}_\alpha$ and $ \alpha\in\left(0,\tilde{\alpha}_V\right)$, \eqref{main-eq} on $\Omega_r$  possess a mountain pass type solution $(\lambda_{r,\alpha},u_{r,\alpha})$ with $u_{r,\alpha}>0$ in $\Omega_r$ and positive energy $E_V(u_{r,\alpha})>0$. Moreover, there exists $C_\alpha>0$ such that
  \[
  \underset{r\to\infty}{\limsup}\,\underset{x\in\Omega_r}{\max}\,u_{r,\alpha}(x)<C_\alpha.
  \]
  \item[(ii)] There exists $0<\bar{\alpha}\le \tilde{\alpha}_V$ such that
  \[
  \underset{r\to\infty}{\liminf}\,\lambda_{r,\alpha}>0\quad\text{for any}\ 0<\alpha\le\bar{\alpha}.
  \]
\end{description}
\end{theorem}

\begin{remark} If $2<p<\frac{10}{3}<\frac{14}{3}<q<6$, we can figure it out that \[
\tilde{\alpha}_V=\bigg(\frac{2(a-\|V_-\|_{\frac{3}{2}}S^{-1})
}{3(p-2)}\bigg)^{\frac{3}{2}}
\bigg(\frac{C_p}{p}A_{p,q}+\frac{C_q}{q}\bigg)^{-\frac{3}{2}}
\bigg(\frac{pC_q}{\beta qC_p}A_{p,q}\bigg)^{\frac{3q-10}{2(q-p)}},
\]
where
\[
A_{p,q}=\frac{(q-2)(3q-10)}{(p-2)(4-3(p-2))}.
\]
\end{remark}

\begin{theorem}\label{betale0-e>0-Omega}
Suppose that $(V_0)$ and $\beta\le0$ hold,   $V(x)\in C^1$ and $\tilde{V}$ is bounded. Then the following hold:
\begin{description}
  \item[(i)] There is $r_\alpha>0$ such that for $\alpha>0$, \eqref{main-eq} on $\Omega_r$ with $r>r_\alpha$ possess a mountain pass type solution $(\lambda_{r,\alpha},u_{r,\alpha})$ with  the following properties: $u_{r,\alpha}>0$ in $\Omega_r$ and positive energy $E_V(u_{r,\alpha})>0$. Moreover, there exists $C_\alpha>0$ such that
      \[
      \underset{r\to\infty}{\limsup}\,\underset{x\in\Omega_r}{\max}\,u_{r,\alpha}(x)<C_\alpha.
      \]
  \item[(ii)] If in addition $\|\tilde{V}_+\|_{\frac{3}{2}}<2aS$ then there exists $\tilde{\alpha}>0$ such that
  \[
  \underset{r\to\infty}{\liminf}\,\lambda_{r,\alpha}>0\quad\text{for any}\ 0<\alpha<\tilde{\alpha}.
  \]
\end{description}
\end{theorem}

For the passage $r\to\infty$ we require the following condition on $V$.
\begin{description}
  \item[$(V_1)$] $V\in C^1$, $\underset{|x|\to\infty}{\lim}\,V(x)=0$, and there exists $\rho\in\left(0,1\right)$ such that
      \[
      \underset{|x|\to\infty}{\liminf}\underset{y\in B(x,\rho|x|)}{\inf}(x\cdot\nabla V(y))e^{\tau|x|}>0\quad\text{for any }\tau>0.
      \]
\end{description}

\begin{theorem}\label{beta>0-e<0-rn-}
Suppose that $V$ satisfies $(V_0)-(V_1)$,
\begin{description}
  \item[(i)] if $\beta>0$, $2<p<\frac{10}{3}<\frac{14}{3}<q<6$ and \eqref{vvv} hold,  then problem \eqref{main-eq} with $\Omega=\mathbb{R}^{3}$ possess for any $0<\alpha<\alpha_V$, $\alpha_V$ as in Theorem \ref{beta>0-e<0-Omega}, a solution $(\lambda_\alpha,u_\alpha)$ with $u_\alpha>0$ , $\lambda_\alpha>0$, and $E_V(u_\alpha)<0$.
  \item[(ii)] if $\beta>0$ and $\tilde{V}$ is bounded, then problem \eqref{main-eq} with $\Omega=\mathbb{R}^{3}$ admits for any $0<\alpha<\bar{\alpha}$, $\bar{\alpha}$ as in Theorem \ref{beta>0-e>0-Omega}, a solution $(\lambda_\alpha,u_\alpha)$ with $u_\alpha>0$, $\lambda_\alpha>0$, and $E_V(u_\alpha)>0$.
  \item[(iii)] if $\beta\le0$, $\tilde{V}$ is bounded and $\|\tilde{V}_+\|_{\frac{3}{2}}<2aS$, then problem \eqref{main-eq} with $\Omega=\mathbb{R}^{3}$ admits for any $0<\alpha<\tilde{\alpha}$, $\tilde{\alpha}$ as in Theorem \ref{betale0-e>0-Omega}(ii), a solution $(\lambda_\alpha,u_\alpha)$ with $u_\alpha>0$, $\lambda_\alpha>0$, and $E_V(u_\alpha)>0$. Moreover, $\underset{\alpha\to0^+}{\liminf}\,E_V(u_\alpha)=\infty$.
\end{description}
\end{theorem}

The rest of this manuscript is organized as follows.
In Section 2, we  collect some notations and some preliminary results   which will be used in this paper.
In Section 3, we obtain the existence and properties of ground states with prescribed mass in large bounded smooth star-shaped domains for the case $\beta>0$.
Section 4 and 5 are devoted to the existence of mountain-pass type solutions for the cases $\beta>0$ and $\beta\le0$, respectively.
The compactness of the normalized solution in $B_r$ as $r\to\infty$  and the proof of Theorem \ref{beta>0-e<0-rn-} are given in the Section 6.

\section{Preliminaries.}
\setcounter{equation}{0}
\setcounter{theorem}{0}	

This section is devoted to collect some preliminary results which will be used in this paper.
Let us first introduce the Gagliardo-Nirenberg inequality, see \cite{Wei82}.
\begin{lemma}
For any $s\in\left(2,6\right)$, there is a constant $C_{s}$ depending on
 s such that
\[
\|u\|_s^s\le C_{s}\|u\|_2^{6-s}\|\nabla u\|_2^{\frac{3(s-2)}{2}},\quad\forall\ u\in H^1(\mathbb{R}^{3}),
\]
where $C_{s}$ be the best constant.
\end{lemma}

Next we recall the Monotonicity trick  \cite{BCJS10}.
\begin{theorem}\label{mp-th}(Monotonicity trick)
Let $(E,\langle\cdot,\cdot\rangle)$ and $(H,(\cdot,\cdot))$ be two infinite-dimensional Hilbert spaces and assume there are continuous injections
\[
E\hookrightarrow H\hookrightarrow E'
\]
Let
\[
\|u\|^2=\langle u,u\rangle,\quad |u|^2=(u,u)\quad\text{for}\  u\in E,
\]
and
\[
S_\mu=\{u\in E:|u|^2=\mu\},\quad T_u S_\mu=\{v\in E:(u,v)=0\}\quad\text{for}\,\mu\in\left(0,+\infty\right).
\]
Let $I\subset\left(0,+\infty\right)$ be an interval and consider a family of $C^2$ functionals $\Phi_\rho: E\to\mathbb{R}$ of the form
\[
\Phi_\rho(u)=A(u)-\rho B(u)\quad\text{for}\,\,\rho \in I,
\]
with $B(u)\ge0$ for every $u\in E$, and
\[
A(u)\to+\infty\quad\text{or}\quad B(u)\to+\infty\quad\text{as} \ \,u\in E\,  \ \text{and}  \  \,\|u\|\to+\infty.
\]
Suppose moreover that $\Phi'_\rho$ and $\Phi''_\rho$ are H\"{o}lder continuous, $\tau\in\left(0,1\right]$, on bounded sets in the following sense: for every $R>0$ there exists $M=M(R)>0$ such that
\begin{equation}
\|\Phi'_\rho(u)-\Phi'_\rho(v)\|\le M\|u-v\|^\tau\quad \|\Phi''_\rho(u)-\Phi''_\rho(v)\|\le M\|u-v\|^\tau
\end{equation}
for every $u,v\in B(0,R)$. Finally, suppose that there exist $w_1,w_2\in S_\mu$ independent of $\rho$ such that
\[
c_\rho:=\underset{\gamma\in\Gamma}{\inf}\underset{t\in\left[0,1\right]}{\max}\Phi_\rho(\gamma(t))
>\max\{\Phi_\rho(w_1),\Phi_\rho(w_2)\}\quad\text{for all}\,\rho\in I,
\]
where
\[
\Gamma=\{\gamma\in C(\left[0,1\right],S_\mu):\gamma(0)=w_1,\gamma(1)=w_2\}.
\]
Then for almost every $\rho\in I$ there exists a sequence $\{u_n\}\subset S_\mu$ such that
\begin{description}
  \item[(i)] $\Phi_\rho(u_n)\to c_\rho$,
  \item[(ii)] $\Phi'_\rho|_{S_\mu}(u_n)\to0$,
  \item[(iii)] $\{u_n\}$ is bounded in $E$.
\end{description}
\end{theorem}

The energy functionals has different geometric structures due to the sign of $\beta$, which leads to need adopt different approaches to investigate the existence and multiplicity of normalized solutions in  $\Omega_r$.
 Consider
\begin{equation}\label{main-eq-omega}
\begin{cases}
-(a+b\int_{\Omega_r}|\nabla u|^2\mathrm{d}x)\Delta u+V(x)u+\lambda u=|u|^{q-2}u+\beta|u|^{p-2}u&\text{in}\ \Omega_r,\\
u\in H^{1}_{0}(\Omega_r),\ \int_{\Omega_r}|u|^2dx=\alpha.
\end{cases}
\end{equation}
where  $2<p<\frac{14}{3}<q<6$, the mass $\alpha>0$ and the parameter $\beta\in\mathbb{R}$ are prescribed, and the frequency $\lambda$ is unknown.
The energy functional $E_r:H_0^1(\Omega_r)\to\mathbb{R}$ is defined by
\[
E_r(u)=\frac{a}{2}\int_{\Omega_r}|\nabla u|^2dx
+\frac{b}{4}\bigg(\int_{\Omega_r}|\nabla u|^2dx\bigg)^2+\frac{1}{2}\int_{\Omega_r}V(x)u^2dx
-\frac{1}{q}\int_{\Omega_r}|u|^{q}dx-\frac{\beta}{p}\int_{\Omega_r}|u|^pdx,
\]
and the mass constraint manifold is defined by
\[
S_{r,\alpha}=\Big\{u\in H_0^1(\Omega_r):\|u\|_2^2=\alpha\Big\}.
\]

\section{Proof of Theorem \ref{beta>0-e<0-Omega}}
\setcounter{equation}{0}
\setcounter{theorem}{0}	
Suppose that the assumptions of Theorem \ref{beta>0-e<0-Omega} hold in this section. Let $0<\alpha<\alpha_V$ be fixed,
 to understand the geometry of the functional $E_r|_{S_{r,\alpha}}$, it is useful to
consider the function $h:\mathbb{R}^+\to\mathbb{R}$
\[
h(t):=\frac{1}{2}\bigg(a-\|V_-\|_{\frac{3}{2}}S^{-1}\bigg)t^2
+\frac{b}{4}t^4-\frac{\beta C_{p}}{p}\alpha^{\frac{6-p}{4}}t^{\frac{3(p-2)}{2}}
-\frac{C_q}{q}\alpha^\frac{6-3q}{4}t^{\frac{3(q-2)}{2}}.
\]
The role of the definition of $\alpha_V$ is clarified by the following lemma.
\begin{lemma}
Under the assumptions of Theorem \ref{beta>0-e<0-Omega}, the function $h$  has a global maximum at positive level. Moreover, there exist three positive constants  $R_1<T_\alpha<R_2$ such that
\[
h(R_1)=h(R_2)=0,\quad h(t)>0\ \text{for}\ R_1<t<R_2,\quad h(T_\alpha)=\max\limits_{t\in\mathbb{R}^+}h(t)>0.
\]
\end{lemma}
\begin{proof}
 Since $\beta>0$ and $2<p<\frac{14}{3}<q<6$, we have that  $h(+\infty)=-\infty$. For $t>0$,
\[
h(t)
\ge t^{\frac{3(p-2)}{2}}\bigg(\phi(t)-\frac{\beta C_{p}}{p}\alpha^{\frac{6-p}{4}}\bigg),
\]
where
\[
\phi(t)=\frac{b}{4}t^{4-\frac{3(p-2)}{2}}
-\frac{C_q}{q}\alpha^{\frac{6-q}{4}}t^{\frac{3(q-p)}{2}}.
\]
It is easy to verify that $\phi$ admits a unique maximum at
\[
\bar{t}=\bigg(\frac{qb(14-3p)}{12C_q(q-p)}\bigg)^{\frac{2}{3q-14}}\alpha^{\frac{q-6}{2(3q-14)}}.
\]
After manipulation and  the definition of $\alpha_V$, we get
\begin{equation}\label{phi-below-bdd}
\phi(\bar{t})>\frac{\beta C_{p}}{p}\alpha^{\frac{6-p}{4}}
\end{equation}
and $h(\bar{t})>0$. From this and $2<p<\frac{14}{3}<q<6$, we further find three positive constants  $R_1<T_\alpha<R_2$ such that $h(R_1)=h(R_2)=0$,
$h(t)>0$ for $t\in\left(R_1,R_2\right)$,
and
\[h(T_\alpha)=\underset{t\in\mathbb{R}^+}{\max}\,h(t)>0.\]
\end{proof}

Now, let us define
\[
\mathcal{B}_{r,\alpha}=\Big\{u\in S_{r,\alpha}:\|\nabla u\|_2^2\le T_\alpha^2\Big\},
\]
from the Pincar\'{e} inequality
\[
\int_{\Omega}|\nabla u|^2\mathrm{d}x\ge\frac{\theta\alpha}{r^2}
\]
for any $u\in S_{r,\alpha}$, where $\theta$ is the principal eigenvalue of $-\Delta$ with Dirichlet boundary condition in $\Omega$. 
There holds $\mathcal{B}_{r,\alpha}=\emptyset$ for $r<\frac{\sqrt{\theta\alpha}}{T_\alpha}$, $\mathcal{B}_{r,\alpha}\neq\emptyset$ for $r\ge\frac{\sqrt{\theta\alpha}}{T_\alpha}$.

\begin{theorem}\label{th-b-e-r-lambda}
Under the assumptions of Theorem \ref{beta>0-e<0-Omega}, there exists a $r_1>0$ such that for $r>r_1$, $e_{r,\alpha}:=\underset{u\in\mathcal{B}_{r,\alpha}}{\inf}E_{r}(u)<0$
  is achieved at $0<u_r\in\mathcal{B}_{r,\alpha}$. Moreover, there is $\lambda_r\in\mathbb{R}$ such that $(\lambda_r,u_r)$ is a solution of \eqref{main-eq-omega}, and $\underset{r\to\infty}{\liminf}\,\lambda_r>0$.
\end{theorem}
\begin{proof}
In view of $2<p<\frac{10}{3}$, there exists $r_0>0$ such that $\frac{1}{2}\bigg(a+\|V\|_{\frac{3}{2}}S^{-1}\bigg)r_0^{-2}\theta\alpha
+\frac{b}{4}r_0^{-4}\theta^2\alpha^2
-\frac{\beta}{p}r_0^{\frac{3(2-p)}{2}}\alpha^{\frac{p}{2}}|\Omega|^{\frac{2-p}{2}}=0$,
and setting
$r_1:=\max\Big\{\sqrt{\theta\alpha}/T_\alpha,r_0\Big\} $.
Next, we construct for $r>r_1$ a function $u_r\in S_{r,\alpha}$ such that $u_r\in \mathcal{B}_{r,\alpha}$ and $E_r(u_r)<0$.
Let $v_1\in S_{1,\alpha}$ be the positive normalized eigenfunction corresponding to $\theta$. Owing to
\[
\int_{\Omega}|\nabla v_1|^2\mathrm{d}x=\theta\alpha
\quad\text{and}\quad
\alpha=\int_{\Omega}|v_1|^2\mathrm{d}x
\le\bigg(\int_{\Omega}|v_1|^q\mathrm{d}x\bigg)^{\frac{2}{q}}|\Omega|^{\frac{q-2}{q}},
\]
let us take $u_0(x):=r^{-\frac{3}{2}}v_1(r^{-1}x)$ for $x\in\Omega_r$. Via direct computation, we conclude that $u_0\in\mathcal{B}_{r,\alpha}$,
\[
\int_{\Omega_r}|\nabla u_0|^2\mathrm{d}x=r^{-2}\theta\alpha
\quad\text{and}\quad
\int_{\Omega_r}|u_0|^p\mathrm{d}x\ge r^{\frac{3(2-p)}{2}}\alpha^{\frac{p}{2}}|\Omega|^{\frac{2-p}{2}}.
\]
From  $2<p<\frac{10}{3}$ and the definition of $r_1$, we have that
\begin{eqnarray*}
E_r(u_0)&=&\frac{a}{2}\int_{\Omega_r}|\nabla u_0|^2\mathrm{d}x
+\frac{1}{2}\int_{\Omega_r}Vu_0^2\mathrm{d}x
+\frac{b}{4}\bigg(\int_{\Omega_r}|\nabla u_0|^2\mathrm{d}x\bigg)^2\\
&&\quad-\frac{1}{q}\int_{\Omega_r}|u_0|^{q}\mathrm{d}x
-\frac{\beta}{p}\int_{\Omega_r}|u_0|^{p}\mathrm{d}x\\
&\le&\frac{1}{2}\bigg(a+\|V\|_{\frac{3}{2}}S^{-1}\bigg)r^{-2}\theta\alpha+\frac{b}{4}r^{-4}\theta^2\alpha^2
-\frac{1}{p}r^{\frac{3(2-p)}{2}}\alpha^{\frac{p}{2}}|\Omega|^{\frac{2-p}{2}}\\
&\le&0
\end{eqnarray*}
for $r>r_1$.
Combining the definition of $e_{r,\alpha}$, we get $e_{r,\alpha}<0$.
On the other hand, the Gagliardo-Nirenberg inequality implies that
\begin{eqnarray}\label{E-GN}
E_r(u_r)&\ge&\frac{1}{2}\bigg(a-\|V_-\|_{\frac{3}{2}}S^{-1}\bigg)\int_{\Omega_r}|\nabla u|^2\mathrm{d}x-\frac{\beta C_{p}}{p}
\alpha^{\frac{2p-3(p-2)}{4}}\bigg(\int_{\Omega_r}|\nabla u|^2\mathrm{d}x\bigg)^{\frac{3(p-2)}{4}}\notag\\
&&\quad-\frac{C_q}{q}\alpha^{\frac{6-q}{4}}\bigg(\int_{\Omega_r}|\nabla u|^2\mathrm{d}x\bigg)^{\frac{3(q-2)}{4}}.
\end{eqnarray}
This indicates that $E_r$ is bounded from below in $\mathcal{B}_{r,\alpha}$. Then from the Ekeland principle, there exists a sequence $\{u_{n,r}\}\subset\mathcal{B}_{r,\alpha}$ satisfying
\[
E_r(u_{n,r})\to e_{r,\alpha},\quad
E'_r(u_{n,r})|_{T_{u_n,r}S_{r,\alpha}}\to0\quad\text{as}\ n\to\infty.
\]
Namely, there exists $u_r\in H_0^1(\Omega_r)$ such that
\[
u_{n,r}\rightharpoonup u_r\ \text{in}\ H_0^1(\Omega_r)\quad u_{n,r}\to u_r\ \text{in}\ L^k(\Omega_r)\quad\text{for }2\le k<6.
\]
Moreover,
\[
\|\nabla u_r\|_2^2\le\underset{n\to\infty}{\liminf}\,\|\nabla u_{n,r}\|_2^2\le T_\alpha^2,
\]
that is $u_r\in\mathcal{B}_{r,\alpha}$.  By Willem \cite[Proposition 5.12]{Wil96}, there exists $\{\lambda_n\}$ such
that
\[
E'_{r}(u_{n,r})+\lambda_nu_{n,r}\to0 \quad\text{in}\ H^{-1}(\Omega_r).
\]
Besides,
\[
\int_{\Omega_r}Vu_{n,r}^2\mathrm{d}x\to\int_{\Omega_r}Vu_{r}^2\mathrm{d}x\quad\text{as}\ n\to\infty,
\]
we arrive that
\begin{eqnarray*}
e_{r,\alpha}+o_n(1)&=&\underset{n\to\infty}{\lim}E_{r}(u_{n,r})\\
&=&\underset{n\to\infty}{\liminf}\Bigg\{\frac{a}{2}\int_{\Omega_r}|\nabla u_{n,r}|^2\mathrm{d}x
+\frac{b}{4}\bigg(\int_{\Omega_r}|\nabla u_{n,r}|^2\mathrm{d}x\bigg)^2
+\frac{1}{2}\int_{\Omega_r}V(x)u_{n,r}^2\mathrm{d}x\\
&&\quad-\frac{1}{q}\int_{\Omega_r}|u_{n,r}|^{q}\mathrm{d}x
-\frac{\beta}{p}\int_{\Omega_r}|u_{n,r}|^p\mathrm{d}x\Bigg\}\\
&\ge&E_{r}(u_r)\ge e_{r,\alpha}
\end{eqnarray*}
 then $u_{n,r}\to u_r$ in $H_0^1(\Omega_r)$ as $n\to\infty$. Consequently, $E_{r}(u_r)=e_{r,\alpha}$ and  $u_r$ is an interior point of $\mathcal{B}_{r,\alpha}$ due to that  $E_r(u)>h(T_\alpha)>0$ for any $u\in\partial\mathcal{B}_{r,\alpha}$ by \eqref{E-GN}.

The Lagrange multiplier theorem implies that there exists $\lambda_r\in\mathbb{R}$ such that $(\lambda_r,u_r)$ is a solution of \eqref{main-eq-omega}.
That is
\begin{eqnarray*}\label{eq-sol-bdd-ge'}
&&(a+b\int_{\Omega_r}|\nabla u|^2\mathrm{d}x)\int_{\Omega_r}|\nabla u_r|^2\mathrm{d}x+\int_{\Omega_r}V(x)u_r^2\mathrm{d}x
+\lambda_r\int_{\Omega_r}|u_r|^2\mathrm{d}x\notag\\
&=&\int_{\Omega_r}|u_r|^{q}\mathrm{d}x
+\beta\int_{\Omega_r}|u_r|^p\mathrm{d}x
.
\end{eqnarray*}
The pohozaev identity leads to
\begin{eqnarray*}\label{pohozaev-ge'}
&&(a+b\int_{\Omega_r}|\nabla u_r|^2\mathrm{d}x)\int_{\Omega_r}|\nabla u_r|^2\mathrm{d}x+\int_{\partial\Omega_r}|\nabla u_r|^2(x\cdot\textbf{n})\mathrm{d}\sigma+\int_{\Omega_r}\tilde{V}u_r^2\mathrm{d}x
\notag\\
&&\quad\quad=-3\int_{\Omega_r}Vu_r^2\mathrm{d}x-3\lambda\int_{\Omega_r}|u_r|^2\mathrm{d}x
+\frac{6}{q}\int_{\Omega_r}|u_r|^{q}\mathrm{d}x
+\frac{6\beta}{p}\int_{\Omega_r}|u_r|^p\mathrm{d}x,
\end{eqnarray*}
where $\textbf{n}$ denotes the outward unit normal vector on $\partial\Omega_r$. Then
\begin{eqnarray*}\label{poho}
&&2(a+b\int_{\Omega_r}|\nabla u_r|^2\mathrm{d}x)\int_{\Omega_r}|\nabla u_r|^2\mathrm{d}x-\int_{\partial\Omega_r}|\nabla u_r|^2(x\cdot\textbf{n})\mathrm{d}\sigma-\int_{\Omega_r}(\nabla V\cdot x)u_r^2\mathrm{d}x\notag\\
&=&\frac{3(q-2)}{q}\int_{\Omega_r}|u_r|^{q}\mathrm{d}x
+\frac{3\beta(p-2)}{p}\int_{\Omega_r}|u_r|^p\mathrm{d}x,
\end{eqnarray*}
Recall that  $\Omega_r$ is starshaped with respect to 0, so $x\cdot\textbf{n}\ge0$ for any $x\in\partial\Omega_r$. From \eqref{vvv}, we have that
\begin{eqnarray*}\label{eralpha-lambda}
\lambda_r\alpha&=&\int_{\Omega_r}|u_r|^{q}\mathrm{d}x
+\beta\int_{\Omega_r}|u_r|^{p}\mathrm{d}x-a\int_{\Omega_r}|\nabla u_{r}|^2\mathrm{d}x
-b\bigg(\int_{\Omega_r}|\nabla u_{r}|^2\mathrm{d}x\bigg)^2
-\int_{\Omega_r}Vu_r^2\mathrm{d}x
\\
&=&\frac{1}{2}\int_{\partial\Omega_r}|\nabla u_r|^2(x\cdot\textbf{n})\mathrm{d}\sigma
+\frac{1}{2}\int_{\Omega_r}(\nabla V\cdot x)u^2\mathrm{d}x-\int_{\Omega_r}Vu_r^2\mathrm{d}x
\\
&&\quad+(1-\frac{3(p-2)}{2p})\beta\int_{\Omega_r}|u_r|^p\mathrm{d}x
+(1-\frac{3(q-2)}{2q})\int_{\Omega_r}|u_r|^{q}\mathrm{d}x\ge m.
\end{eqnarray*}
Hence, $\liminf\limits_{r\to\infty}\,\lambda_r>0$.
Moreover, the strong maximum principle implies $u_r>0$.
\end{proof}

The following is the priori bound to the solutions of problem \eqref{main-eq-omega}.
\begin{lemma}\label{ur-ubdd}
Let $\{(\lambda_r,u_r)\}$ be  nonnegative solutions of  \eqref{main-eq-omega}  with $\|u_r\|_{H^1}\le C$, where $C>0$ is independent of $r$, then $\limsup\limits_{r\to\infty}\,\|u_r\|_\infty<\infty$.
\end{lemma}

\begin{proof}
    We assume by the contradiction  that there are $\{u_r\}$ and $x_r\in\Omega_r$ such that
    \[
    M_r:=\underset{x\in\Omega_r}{\max}\,u_r(x)=u_r(x_r)\to\infty\quad\text{as}\ r\to\infty.
    \]
    Now we perform a rescaling, defining
    \[
\omega_r:=\frac{u_r(x_r+\tau_r x)}{M_r}\quad\text{for}\ x\in\Sigma^r:=\{x\in\mathbb{R}^3:x_r+\tau_r x\in\Omega_r\},
    \]
    where $\tau_r=M_r^{-\frac{q-2}{2}}$. From the uniform boundedness of $\|\nabla u_r\|_2$ and the boundedness is independent of $r$, we set $\|\nabla u_r\|_2\to\Lambda$ as $r\to\infty$.  Through a delicate calculation, we get that $\tau_r\to0$, $\|\omega_r\|_{L^\infty(\Sigma^r)}\le1$ and $\omega_r$ satisfies
    \begin{align}\label{equ:241204-e1}
        -(a+b\Lambda^2)\Delta\omega_r+\tau_r^2(V(x_r+\tau_rx)+\lambda_r)\omega_r=|\omega_r|^{q-2}\omega_r+\beta\tau_r^{\frac{2(q-p)}{q-2}}|\omega_r|^{p-2}\omega_r\quad\text{in}\ \Sigma^r.
    \end{align}
It follows from \eqref{main-eq-omega},  the Gagliardo-Nirenberg inequality,  the Sobolev inequatlity and $\left\|u_r\right\|_{H^1} \leq C$ that  the sequence $\left\{\lambda_r\right\}$ is bounded. According to   $L^p$ estimates  (see [20, Theorem 9.11]), we can deduce that $\omega_r \in W_{\text {loc}}^{2, p}(\Sigma)$ and $\|\omega_r\|_{W_{\text {loc }}^{2, p}(\Sigma)} \leq C$ for any $p>1$ and $\Sigma:=\lim\limits _{r \rightarrow \infty} \Sigma^r$. Using Sobolev embedding theorem, we can obtain that $\omega_r \in C_{\mathrm{loc}}^{1, \beta}(\Sigma)$ for some $\beta \in(0,1)$ and $\|\omega_r\|_{C_{\text {loc }}^{1, \beta}(\Sigma)} \leq C$.   Therefore,  applying  the Arzela-Ascoli theorem, we find that there exists $\omega$ such that up to a subsequence
 \[\omega_r\to\omega\quad\text{in}\  C_{loc}^\beta(\Sigma).\]
 where $\Sigma=\lim_{r\to\infty}\Sigma^r$ is a smooth domain.
By the standard direct method ( see \cite[Lemma2.7]{BQZ24}),  we obtain that   \[
\liminf_{r\to\infty}\frac{\text{dist}(x_r,\partial\Omega_r)}{\tau_r}>0.\]
As a result, Let $r \to \infty$ in \eqref{equ:241204-e1}, we find that
 $\omega$  is a nonnegative solution of problem
    \[\begin{cases}
        -\Delta\omega =\frac{1}{a+b\Lambda^2}|\omega|^{q-2}\omega,&\text{in}\ \Sigma,\\
        \omega=0&\text{on}\ \partial\Sigma.
    \end{cases}\]
 If \[
\liminf_{r\to\infty}\frac{\text{dist}(x_r,\partial\Omega_r)}{\tau_r}=\infty\]
occurs,  $\Sigma=\mathbb{R}^{3}$. Since $q<6$, by the remarkable result \cite{CL91}, the only nonnegative solution of is 0,  which is impossible due to the fact that $\omega(0)=\liminf_{r\to\infty}\omega_r(0)=1$. If \[
\liminf_{r\to\infty}\frac{\text{dist}(x_r,\partial\Omega_r)}{\tau_r}=d>0\]
occurs, we next \textbf{claim} $\Sigma$ is a half space. 
In fact, let  dist$(x_r,\partial\Omega_r)=|x_r-z_r|$ with $z_r\in\partial\Omega_r$, then $\tilde{z}_r=z_r/r\in\partial\Omega$ and $\tilde{x}_r=x_r/r\in\Omega$.  Based on the definition of $\Sigma_r$,  the origin is located at $\tilde{x}_r$ and $\tilde{x}_r-\tilde{z}_r=|x_r-z_r|(1,0,...,0)$ by translating  and rotating  the coordinate
system.
Assume up to a subsequence that $\tilde{z}_r\to z$ with $z\in\Omega$.
And by the smoothness of the domain see \cite[section 6.2]{GT83},  functions $f_r,f:\mathbb{R}^{2}\to\mathbb{R}$ are smooth  such that
\[f_r(0)=-|\tilde{x}_r-\tilde{z}_r|=-\frac{|x_r-z_r|}{r},\quad \nabla f_r(0)=0,\]
    \begin{equation}\label{>f}
        \Omega\cap B(\tilde{z}_r,\delta)=\{x\in B(\tilde{z}_r,\delta):x_1>f_r(x_2,x_3)\},
    \end{equation}
    and
    \begin{equation}\label{=f}
        \partial\Omega\cap B(\tilde{z}_r,\delta)=\{x\in B(\tilde{z}_r,\delta):x_1=f_r(x_2,x_3)\}.
    \end{equation}
    Moreover, \eqref{>f} and \eqref{=f} hold by replacing $\tilde{z}_r$ and $f_r$ with z and f , respectively.
And $f_r\to f$ in $C^1(B_{\frac{\delta}{2}})$, where $B_{\frac{\delta}{2}}\subset\mathbb{R}^{2}$. Consequently,
\[
 \partial\Omega_r\cap B(z_r,r\delta)=\{(rf_r(x'),rx'):|x'|<\delta\}=\big\{\bigg(rf_r(\frac{x'}{r}),x'\bigg):|x'|<r\delta\big\},
\]
where $x'\in\mathbb{R}^{2}$. As a consequence, for any $x'\in\mathbb{R}^{2}$ and large $r$,
\[
y_r=\bigg(rf_r(\frac{\tau_rx'}{r}),\tau_rx'\bigg)\in\partial\Omega_r\cap B(z_r,r\delta).
\]
Moreover,
\[
\frac{y_r-x_r}{\tau_r}=\bigg(\frac{rf_r(\frac{\tau_rx'}{r})}{\tau_r},x'\bigg).
\]
Note that
\begin{eqnarray*}
\frac{rf_r(\frac{\tau_rx'}{r})}{\tau_r}&=&
\frac{rf_r(\frac{\tau_rx'}{r})-rf_r(0)+rf_r(0)}{\tau_r}
=\frac{r\nabla f_r(\frac{\theta_r\tau_rx'}{r})\cdot\frac{\tau_rx'}{r}-|x_r-z_r|}{\tau_r}  \\
&\to&\nabla f(0)\cdot x'-d=-d.
\end{eqnarray*}
As a result, $\frac{y_r-x_r}{\tau_r}\to(-d,x')$ as $r\to\infty$. Therefore $\Sigma=\{x\in\mathbb{R}^{3}:x_1>-d\}$, the claim hold. Then we can use the Liouville theorem \cite{EL82}, $\omega=0$ in $\Sigma$, this contradicts $\omega(0)=\underset{r\to\infty}{\liminf}\,\omega_r(0)=1$.  The proof is now complete.
\end{proof}

\begin{remark}\label{ur-ubdd-beta}
Note that the proof of Lemma \ref{ur-ubdd}  does not depend on $\beta$.
\end{remark}

\noindent\textbf{Proof of Theorem  \ref{beta>0-e<0-Omega}:} The proof is an immediate consequence of Theorem \ref{th-b-e-r-lambda}  and  Lemma \ref{ur-ubdd}.

\section{Proof of Theorem \ref{beta>0-e>0-Omega}}
\setcounter{equation}{0}
\setcounter{theorem}{0}	
This section is devoted to the existence of mountain-pass type solution for the case $\beta>0$.
Under the assumptions of Theorem \ref{beta>0-e>0-Omega}, for  $s\in\left[1/2,1\right]$ we define the functional $J_{r,s}:S_{r,\alpha}\to\mathbb{R}$ by
\[
J_{r,s}(u)=\frac{a}{2}\int_{\Omega_r}|\nabla u|^2\mathrm{d}x+\frac{b}{4}\bigg(\int_{\Omega_r}|\nabla u|^2\mathrm{d}x\bigg)^2
+\frac{1}{2}\int_{\Omega_r}Vu^2\mathrm{d}x
-s\,\bigg(\frac{1}{q}\int_{\Omega_r}|u|^{q}\mathrm{d}x
+\frac{\beta}{p}\int_{\Omega_r}|u|^{p}\mathrm{d}x\bigg).
\]
Note that if $u\in S_{r,\alpha}$ is a critical point of $J_{r,s}$, then there is $\lambda\in\mathbb{R}$ such that $(\lambda,u)$ is a solution of the problem
\begin{equation}\label{main-eq-s-omega-2}
\begin{cases}
-(a+b\int_{\Omega_r}|\nabla u|^2dx)\Delta u+V(x)u+\lambda u=s|u|^{q-2}u+s\beta|u|^{p-2}u&\text{in}\ \Omega_r,\\
u\in H^{1}_{0}(\Omega_r),\ \int_{\Omega_r}|u|^2dx=\alpha.
\end{cases}
\end{equation}

\begin{lemma}\label{mp-betage0}
For any $0<\alpha<\tilde{\alpha}_V$, where $\tilde{\alpha}_V$ is defined in Theorem \ref{beta>0-e>0-Omega}. There exist $\tilde{r}_\alpha, t_\alpha>0$ and $F_\alpha>0$ such that the following hold.
\begin{description}
  \item[(i)] There are $u^0,u^1\in S_{\tilde{r}_\alpha,\alpha}$ such that $J_{r,s}(u^1)\le0$ for any $r>\tilde{r}_\alpha$, $s\in\left[\frac{1}{2},1\right]$,
  $\|\nabla u^0\|_2^2<t_\alpha<\|\nabla u^1\|_2^2$ and $J_{r,s}(u^0)<F_\alpha$.
  \item[(ii)] If $u\in S_{r,\alpha}$ satisfies $\|\nabla u\|_2^2=t_\alpha$,
  then there holds $J_{r,s}(u)\ge F_\alpha$.
  \item[(iii)] Set
  \[
  m_{r,s}(\alpha)=\underset{\gamma\in\Gamma_{r,\alpha}}{\inf}
  \underset{t\in\left[0,1\right]}{\sup}J_{r,s}(\gamma(t)),
  \]
  with
  \[
  \Gamma_{r,\alpha}=\Big\{\gamma\in C(\left[0,1\right],S_{r,\alpha}):\gamma(0)=u^0,\gamma(1)=u^1\Big\}.
  \]
  Then $F_\alpha\le m_{r,s}(\alpha)\le  H_\alpha:=\underset{t\in\mathbb{R}^+}{\max}\,h(t)$,
  where  $h:\mathbb{R}^+\to\mathbb{R}$ is defined by
\[
h(t)=\frac{1}{2}\bigg(a+\|V\|_{\frac{3}{2}}S^{-1}\bigg)t^2\theta\alpha+\frac{b}{4}t^4\theta^2\alpha^2
-\frac{1}{2q}t^{\frac{3(q-2)}{2}}\alpha^{\frac{q}{2}}|\Omega|^{\frac{2-q}{2}}.
\]
\end{description}
Here $\theta$ is the positive normalized eigenfunction of $-\Delta$ with Dirichlet boundary condition in $\Omega$.
Moreover, if $2<p<\frac{10}{3}$, we can figure out that
\[
t_\alpha=\bigg(\frac{2(a-\|V_-\|_{\frac{N}{2}}S^{-1})}{3(q-2)A}
  \bigg)^{\frac{4}{3q-10}}\alpha^{\frac{q-6}{3q-10}}
\]
and
\[
F_\alpha=\frac{3q-10}{4}\bigg(\frac{2(a-\|V_-\|_{\frac{N}{2}}S^{-1})}{3(q-2)}
  \bigg)^{\frac{3q-6}{3q-10}}\alpha^{\frac{q-6}{3q-10}}
A^{\frac{4}{10-3q}},
\]
where
  \[
  A=\bigg(\frac{C_q(q-2)(3(q-2)-4)}{ q(p-2)(4-3(p-2))}\bigg).
  \]
\end{lemma}

\begin{proof}
 Clearly the set $S_{r,\alpha}$ is path connected.  Let  $v_t(x)=t^{\frac{3}{2}}v_1(tx)$, where $v_1\in S_{1,\alpha}$ is the positive normalized eigenfunction of $-\Delta$ with Dirichlet boundary condition in $\Omega$ associated to $\theta$.
Owing to
\[
\int_{\Omega}|\nabla v_1|^2\mathrm{d}x=\theta\alpha\quad\text{and}\quad\int_{\Omega}|v_1|^q\mathrm{d}x
\ge\alpha^{\frac{q}{2}}|\Omega|^{\frac{2-q}{2}},
\]
for $x\in\Omega_{\frac{1}{t}}$ and $t>0$ there holds
\begin{eqnarray}\label{j-vt-s-h}
J_{\frac{1}{t},s}(v_t)&\le&\frac{1}{2}t^2\bigg(a+\|V\|_{\frac{3}{2}}S^{-1}\bigg)\int_{\Omega}|\nabla v_1|^2\mathrm{d}x
+\frac{b}{4}t^4\bigg(\int_{\Omega}|\nabla v_1|^2\mathrm{d}x\bigg)^2\notag\\
&&\quad-\frac{\beta}{2p}t^{\frac{3(p-2)}{2}}\int_{\Omega}|v_1|^p\mathrm{d}x
-\frac{1}{2q}t^{\frac{3(q-2)}{2}}\int_{\Omega}|v_1|^{q}\mathrm{d}x\\
&\le& h(t)\notag.
\end{eqnarray}
Via direct computation, there is $t_0$ such that  $h(t)<0$ for   $t>t_0$ and $h(t)>0$ for   $0<t<t_0$.
Then we obtain that for any $r\ge\frac{1}{t_0}$ and $s\in\left[\frac{1}{2},1\right]$,
\begin{equation}\label{j-rs-le0}
J_{r,s}(v_{t_0})=J_{\frac{1}{t_0},s}(v_{t_0})\le h(t_0)=0.
\end{equation}
Additional, we assume that the function $h$ attains its maximum at $t_\alpha$,
and so we find  $t_1\in\left(0,t_\alpha\right)$ such that for any $t\in\left[0,t_1\right]$,
\begin{equation}\label{j-h-bdd}
h(t)<h(t_1)\le F_\alpha,
\end{equation}
where $F_\alpha$ will be given in the later.
 For $\frac{10}{3}<p<\frac{14}{3}<q<6$, the Sobolev inequality and the Gagliardo-Nirenberg inequality implies that
\begin{eqnarray}\label{jrs-bel'}
J_{r,s}(u)&\ge&\frac{b}{4}\bigg(\int_{\Omega_r}|\nabla u|^2dx\bigg)^2
-\frac{C_{p}\beta\alpha^{\frac{6-p}{4}}}{p}\bigg(\int_{\Omega_r}|\nabla u|^2dx\bigg)^{\frac{3(p-2)}{4}}\notag\\
&&\quad-\frac{C_q\alpha^{\frac{6-q}{4}}}{q}\bigg(\int_{\Omega_r}|\nabla u|^2\mathrm{d}x\bigg)^{\frac{3(q-2)}{4}}.
\end{eqnarray}
Let $f:\mathbb{R}^+\to\mathbb{R}$ be defined by
\[
f(t):=\frac{b}{4}t^2
-\frac{C_{p}\beta\alpha^{\frac{6-p}{4}}}{p}t^{\frac{3(p-2)}{4}}
-\frac{C_q\alpha^{\frac{6-q}{4}}}{q}t^{\frac{3(q-2)}{4}},
\]
 there exist three positive constants $l_1<t_\alpha<l_2$ such that
\[
f(t)<0,\ t\in\left(0,l_1\right)\cup\left(l_2,\infty\right),\quad\ f(t)>0,\ \ t\in\left(l_1,l_2\right)
\ \text{and}\ F_\alpha:=f(t_\alpha)=\max_{t\in\mathbb{R}^+}f(t)>0.
\]
Selecting $\bar{r}_\alpha=\max\{\frac{1}{t_1},\sqrt{\frac{2\theta\alpha}{t_\alpha}}\}$, we derive that  $v_{\frac{1}{\bar{r}_\alpha}}\in S_{\bar{r}_\alpha,\alpha}\subset S_{r,\alpha}$ for $r>\bar{r}_{\alpha}$, and
\begin{equation}\label{nabla-v-ralpha-bdd'}
\|\nabla v_{\frac{1}{\bar{r}_\alpha}}\|_2^2
=\bigg(\frac{1}{\bar{r}_\alpha}\bigg)^2\|\nabla v_1\|_2^2
<t_\alpha,
\end{equation}
moreover,
\begin{equation}\label{j-h-t1-bdd'}
J_{\bar{r}_\alpha,s}\bigg(v_{\frac{1}{\bar{r}_\alpha}}\bigg)\le h\bigg(\frac{1}{{\bar{r}_\alpha}}\bigg)
\le h(t_1)\le F_\alpha.
\end{equation}
Now we define $u^0:=v_{\frac{1}{\bar{r}_\alpha}}$, $u^1:=v_{t_0}$ and
$\tilde{r}_\alpha:=\max\Big\{\frac{1}{t_0},\bar{r}_\alpha\Big\}$.
Thanks to \eqref{j-rs-le0}, \eqref{j-h-bdd}, \eqref{nabla-v-ralpha-bdd'} and \eqref{j-h-t1-bdd'}, we derive (i) holds. Furthermore, by employing \eqref{jrs-bel'}, we also infer (ii) holds.

For  $2<p<\frac{10}{3}, \frac{14}{3}<q<6$,   from \eqref{j-vt-s-h}, we also obtain  $t_1'\in\left(0,t_\alpha\right)$ such that for any $t\in\left[0,t_1'\right]$,
\begin{equation}\label{j-h-bdd'}
h(t)<h(t_1')\le \frac{3q-10}{4}\bigg(\frac{2(a-\|V_-\|_{\frac{3}{2}}S^{-1})}{3(q-2)}
  \bigg)^{\frac{3q-6}{3q-10}}\alpha^{\frac{q-6}{3q-10}}
A^{\frac{4}{10-3q}}.
\end{equation}
The Sobolev inequality and the Gagliardo-Nirenberg inequality implies that
\begin{eqnarray}\label{jrs-bel}
J_{r,s}(u)&\ge&\frac{1}{2}\bigg(a+\|V\|_{\frac{3}{2}}S^{-1}\bigg)\int_{\Omega}|\nabla u|^2\mathrm{d}x
-\frac{C_{p}\beta\alpha^{\frac{6-p}{4}}}{p}\bigg(\int_{\Omega_r}|\nabla u|^2dx\bigg)^{\frac{3(p-2)}{4}}\notag\\
&&\quad-\frac{C_q\alpha^{\frac{6-q}{4}}}{q}\bigg(\int_{\Omega_r}|\nabla u|^2\mathrm{d}x\bigg)^{\frac{3(q-2)}{4}}.
\end{eqnarray}
Let $f:\mathbb{R}^+\to\mathbb{R}$ be defined by
\[
f(t):=\frac{1}{2}\bigg(a+\|V\|_{\frac{3}{2}}S^{-1}\bigg)t
-\frac{C_{p}\beta\alpha^{\frac{6-p}{4}}}{p}t^{\frac{3(p-2)}{4}}
-\frac{C_q\alpha^{\frac{6-q}{4}}}{q}t^{\frac{3(q-2)}{4}}.
\]
Then there exist three positive constants $l_1<l_M<l_2$ such that
\[
f(t)<0\ \ \text{for}\ t\in\left(0,l_1\right)\cup\left(l_2,\infty\right),\quad\ f(t)>0\ \ \text{for}\ t\in\left(l_1,l_2\right),
\quad\text{and}\ f(l_M)=\max_{t\in\mathbb{R}^+}f(t)>0.
\]
After a direct calculation, we deduce $f''(t)\le0$ if and only if $t>t_2$ with
\[
t_2=\bigg(\frac{C_{p}\beta q(p-2)(4-3(p-2))}{pC_q(q-2)(3(q-2)-4)}
\bigg)^{\frac{4}{3(q-p)}}\alpha^{\frac{1}{3}},
\]
and we further deduce
\[
\underset{t\in\mathbb{R}^+}{\max}f(t)=\underset{t\in\left[t_2,\infty\right)}{\max}f(t).
\]
For any $t\ge t_2$, it holds
\begin{eqnarray}\label{jrs-bel}
f(t)&=&\frac{1}{2}\bigg(a-\|V_-\|_{\frac{3}{2}}S^{-1}\bigg)t-\frac{C_{p}\beta\alpha^{\frac{q-p}{4}}
\alpha^{\frac{6-q}{4}}}{p}t^{\frac{3(p-2)}{4}}-\frac{C_q\alpha^{\frac{6-q}{4}}}{q}t^{\frac{3(q-2)}{4}}\notag\\
&\ge&\frac{1}{2}\bigg(a-\|V_-\|_{\frac{3}{2}}S^{-1}\bigg)t
-\bigg(\frac{C_q(q-2)(3(q-2)-4)}{q(p-2)(4-3(p-2))}+\frac{C_q}{q}\bigg)t^{\frac{3(q-2)}{4}}\alpha^{\frac{6-q}{4}}\notag\\
&=:&g(t).
\end{eqnarray}
Let
\[
A=\bigg(\frac{C_q(q-2)(3(q-2)-4)}{ q(p-2)(4-3(p-2))}\bigg)
\]
and
\[
t_g=\bigg(\frac{2(a-\|V_-\|_{\frac{N}{2}}S^{-1})}{3(q-2)A}
  \bigg)^{\frac{4}{3q-10}}\alpha^{\frac{q-6}{3q-10}},
\]
so that $t_g>t_2$ by the definition of $\tilde{\alpha}_V$,  $\underset{t\in\left[t_2,\infty\right)}{\max}\,g(t)=g(t_g)$ and
\[
\underset{t\in\mathbb{R}^+}{\max}f(t)
=\underset{t\in\left[t_2,\infty\right)}{\max}g(t)
=\frac{3q-10}{4}\bigg(\frac{2(a-\|V_-\|_{\frac{3}{2}}S^{-1})}{3(q-2)}
  \bigg)^{\frac{3q-6}{3q-10}}\alpha^{\frac{q-6}{3q-10}}
A^{\frac{4}{10-3q}}.
\]
Selecting $\bar{r}_\alpha=\max\{\frac{1}{t_1'},\sqrt{\frac{2\theta\alpha}{t_g}}\}$, we derive that  $v_{\frac{1}{\bar{r}_\alpha}}\in S_{\bar{r}_\alpha,\alpha}\subset S_{r,\alpha}$ for $r>\bar{r}_{\alpha}$, and
\begin{equation}\label{nabla-v-ralpha-bdd}
\|\nabla v_{\frac{1}{\bar{r}_\alpha}}\|_2^2
=\bigg(\frac{1}{\bar{r}_\alpha}\bigg)^2\|\nabla v_1\|_2^2
<\bigg(\frac{2(a-\|V_-\|_{\frac{3}{2}}S^{-1})}{3(q-2)A}
  \bigg)^{\frac{4}{3q-10}}\alpha^{\frac{q-6}{3q-10}},
\end{equation}
moreover,
\begin{equation}\label{j-h-t1-bdd}
J_{\bar{r}_\alpha,s}\bigg(v_{\frac{1}{\bar{r}_\alpha}}\bigg)\le h\bigg(\frac{1}{{\bar{r}_\alpha}}\bigg)
\le h(t'_1).
\end{equation}
Now we define $u^0:=v_{\frac{1}{\bar{r}_\alpha}}$, $u^1:=v_{t_0}$ and
\[
\tilde{r}_\alpha:=\max\Big\{\frac{1}{t_0},\bar{r}_\alpha\Big\}.
\]
Thanks to \eqref{j-rs-le0}-\eqref{j-h-bdd} and \eqref{nabla-v-ralpha-bdd}-\eqref{j-h-t1-bdd}, we derive (i) holds. Furthermore, by employing \eqref{jrs-bel}, we also infer (ii) holds.

For (iii), since $J_{r,s}(u^1)\le0$ for any $\gamma\in\Gamma_{r,\alpha}$, we have
\[
\|\nabla \gamma(0)\|_2^2<t_g<\|\nabla\gamma(1)\|_2^2.
\]
It then follows from \eqref{jrs-bel} that
\[
\underset{t\in\left[0,1\right]}{\max}J_{r,s}(\gamma(t))\ge F_\alpha
\]
for any $\gamma\in\Gamma_{r,\alpha}$. Now we define a path $\gamma:\left[0,1\right]\to S_{r,\alpha}$ by
\[
\gamma(\tau):\Omega_r\to\mathbb{R},\quad x\mapsto\bigg(\tau t_0+(1-\tau)\frac{1}{\tilde{r}_\alpha}\bigg)^{\frac{3}{2}}v_1\bigg(\bigg(\tau t_0+(1-\tau)\frac{1}{\tilde{r}_\alpha}\bigg)x\bigg).
\]
Clearly $\gamma\in\Gamma_{r,\alpha}$. Then (iii) holds.
\end{proof}

\begin{theorem}\label{beta>0-s-omega-ge}
Let $r>\tilde{r}_\alpha$, where $\tilde{r}_\alpha$ is defined in Lemma \ref{mp-betage0}. Problem \eqref{main-eq-s-omega-2} admits a solution $(\lambda_{r,s},u_{r,s})$ for almost every $s\in\left[\frac{1}{2},1\right]$ and $0<\alpha<\tilde{\alpha}_V$. Moreover, there hold $u_{r,s}>0$ and $J_{r,s}(u_{r,s})=m_{r,s}(\alpha)$.
\end{theorem}
\begin{proof}
For fixed $\alpha\in\left(0,\tilde{\alpha}_V\right)$, let us apply Theorem \ref{mp-th} to $E_{r,s}$ with $\Gamma_{r,\alpha}$ given in Lemma \ref{mp-betage0}-(iii),
\[
A(u)=\frac{a}{2}\int_{\Omega}|\nabla u|^2\mathrm{d}x+\frac{b}{4}\bigg(\int_{\Omega}|\nabla u|^2\mathrm{d}x\bigg)^2+\frac{1}{2}\int_{\Omega}V(x)u^2\mathrm{d}x
-\frac{\beta}{p}\int_{\Omega}|u|^p\mathrm{d}x
\]
and
\[
B(u)=\frac{1}{q}\int_{\Omega}|u|^{q}\mathrm{d}x.
\]
Thanks to Lemma \ref{mp-betage0}, the assumptions in Theorem \ref{mp-th} hold. Hence, for almost every $s\in\left[\frac{1}{2},1\right]$, there exists a nonnegative bounded Palais-Smale sequence $\{u_n\}$:
\[
J_{r,s}(u_n)\to m_{r,s}(\alpha)\quad\text{and}\quad J'_{r,s}(u_n)|_{T_{u_n}S_{r,\alpha}}\to0,
\]
where $T_{u_n}S_{r,\alpha}$ denoted the tangent space of $S_{r,\alpha}$ at $u_n$. Note that
\begin{equation}\label{un-ers-ge}
J'_{r,s}(u_n)+\lambda_nu_n\to0 \quad\text{in}\ H^{-1}(\Omega_r)
\end{equation}
and
\[
\lambda_n=-\frac{1}{\alpha}\bigg((a+b\int_{\Omega_r}|\nabla u_n|^2\mathrm{d}x)\int_{\Omega_r}|\nabla u_n|^2\mathrm{d}x+\int_{\Omega_r}V(x) u^2_n\mathrm{d}x-\beta\int_{\Omega_r}|u_n|^p\mathrm{d}x-s\int_{\Omega_r}|u_n|^{q}\mathrm{d}x\bigg).
\]
is bounded.
Furthermore, let us assume that $\int_{\Omega_r}|\nabla u_n|^2\mathrm{d}x\to\Lambda$, there exist $u_0\in H_0^1(\Omega_r)$ and $\lambda\in\mathbb{R}$ such that up to a subsequence,
\[
\lambda_n\to\lambda,\quad u_n\rightharpoonup u_0\ \text{in}\ H_0^1(\Omega_r)\quad \text{and}\ u_n\to u_0\ \text{in}\ L^t(\Omega_r)\ \text{for all }2\le t<6,
\]
and $u_0$ satisfies
\begin{equation}\label{eq-u0-ge}
\begin{cases}
-(a+b\Lambda^2)\Delta u_0+Vu_0+\lambda u_0=s|u_0|^{q-2}u_0+\beta|u_0|^{p-2}u_0&\text{in}\,\Omega_r,\\
u_0\in H_0^1(\Omega_r),\quad\int_{\Omega_r}|u_0|^2\mathrm{d}x=\alpha.
\end{cases}
\end{equation}
In view of \eqref{un-ers-ge}, we have
\[
J'_{r,s}(u_n)u_0+\lambda_n\int_{\Omega_r}u_nu_0\mathrm{d}x\to0\quad\text{and}\quad J'_{r,s}(u_n)u_n+\lambda_n\alpha\to0\quad\text{as}\quad n\to\infty,
\]
and
\[
J'_{r,s}(u_0)u_n+\lambda\int_{\Omega_r}u_nu_0\mathrm{d}x\to0,\quad J'_{r,s}(u_0)u_0+\lambda\alpha\to0\quad\text{as}\quad n\to\infty.
\]
Owing to
\[
\underset{n\to\infty}{\lim}\int_{\Omega_r}V(x)u_n^2\mathrm{d}x=\int_{\Omega_r}V(x)u_0^2\mathrm{d}x.
\]
We conclude that
\begin{eqnarray*}
o_n(1)&=&\big(J'_{r,s}(u_n)-J'_{r,s}(u_0)\big)(u_n-u_0)\notag\\
&=&(a+b\int_{\Omega_r}|\nabla u_n|^2\mathrm{d}x)\int_{\Omega_r}\nabla u_n\nabla(u_n-u_0)\mathrm{d}x\notag\\
&&\quad-(a+b\Lambda)\int_{\Omega_r}\nabla u_0\nabla(u_n-u_0)\mathrm{d}x
-\int_{\Omega_r}V(x)(u_n^2-u_0^2)\mathrm{d}x\notag\\
&=&(a+b\int_{\Omega_r}|\nabla u_n|^2\mathrm{d}x)(\int_{\Omega_r}\nabla u_n\nabla(u_n-u_0)\mathrm{d}x-\int_{\Omega_r}\nabla u_0\nabla(u_n-u_0)\mathrm{d}x)\notag\\
&&\quad+b(\int_{\Omega_r}|\nabla u_n|^2\mathrm{d}x-\Lambda)\int_{\Omega_r}\nabla u_0\nabla(u_n-u_0)\mathrm{d}x\notag\\
&\ge&a\int_{\Omega_r}|\nabla(u_n-u_0)|^2\mathrm{d}x.
\end{eqnarray*}
 Hence,  $u_n\to u_0$ in $H_0^1(\Omega_r)$ as $n\to\infty$. Consequently, $J_{r,s}(u_0)=m_{r,s}(\alpha)$ and $u_0$ is a nonnegative normalized solution to \eqref{main-eq-s-omega-2}.
\end{proof}

\begin{lemma}\label{nabla-u-bdd-j}
For fixed $\alpha>0$ the set of solutions $u\in S_{r,\alpha}$ of \eqref{main-eq-s-omega-2} is bounded uniformly in $s$ and $r$.
\end{lemma}
\begin{proof}
Based on the fact that $(\lambda,u)\in\mathbb{R}\times S_{r,\alpha}$ is a solution of problem \eqref{main-eq-s-omega-2}, we know
\begin{eqnarray}\label{eq-sol-bdd-ge}
&&(a+b\int_{\Omega_r}|\nabla u|^2\mathrm{d}x)\int_{\Omega_r}|\nabla u|^2\mathrm{d}x+\int_{\Omega_r}V(x)u^2\mathrm{d}x
+\lambda\int_{\Omega_r}|u|^2\mathrm{d}x\notag\\
&=&s\int_{\Omega_r}|u|^{q}\mathrm{d}x
+s\beta\int_{\Omega_r}|u|^p\mathrm{d}x
.
\end{eqnarray}
In addition, the pohozaev identity leads to
\begin{eqnarray}\label{pohozaev-ge}
&&\frac{(a+b\int_{\Omega_r}|\nabla u|^2\mathrm{d}x)}{6}\int_{\Omega_r}|\nabla u|^2\mathrm{d}x+\frac{1}{6}\int_{\partial\Omega_r}|\nabla u|^2(x\cdot\textbf{n})\mathrm{d}\sigma+\frac{1}{6}\int_{\Omega_r}\tilde{V}u^2\mathrm{d}x
\notag\\
&&\quad\quad=-\frac{1}{2}\int_{\Omega_r}Vu^2\mathrm{d}x-\frac{\lambda}{2}\int_{\Omega_r}|u|^2\mathrm{d}x
+\frac{s}{q}\int_{\Omega_r}|u|^{q}\mathrm{d}x
+\frac{s\beta}{p}\int_{\Omega_r}|u|^p\mathrm{d}x,
\end{eqnarray}
where $\textbf{n}$ denotes the outward unit normal vector on $\partial\Omega_r$.
Then via the inequality \eqref{eq-sol-bdd-ge} and \eqref{pohozaev-ge}, we deduce that
\begin{eqnarray*}
&&\frac{(a+b\int_{\Omega_r}|\nabla u|^2\mathrm{d}x)}{3}\int_{\Omega_r}|\nabla u|^2\mathrm{d}x-\frac{1}{6}\int_{\partial\Omega_r}|\nabla u|^2(x\cdot\textbf{n})\mathrm{d}\sigma-\frac{1}{6}\int_{\Omega_r}(\nabla V\cdot x)u^2\mathrm{d}x\\
&=&\frac{(q-2)s}{2q}\int_{\Omega_r}|u|^{q}\mathrm{d}x+\frac{\beta(p-2)s}{2p}\int_{\Omega_r}|u|^p\mathrm{d}x\\
&\ge&\frac{q-2}{2}\bigg(\frac{s}{q}\int_{\Omega_r}|u|^{q}\mathrm{d}x
+\frac{\beta s}{p}\int_{\Omega_r}|u|^p\mathrm{d}x\bigg)
+\frac{\beta s(p-q)}{2p}\int_{\Omega_r}|u|^p\mathrm{d}x\\
&=&\frac{q-2}{2}\bigg(\frac{1}{2}\int_{\Omega_r}|\nabla u|^2\mathrm{d}x
+\frac{1}{2}\int_{\Omega_r}V|u|^2\mathrm{d}x-m_{r,s}(\alpha)\bigg)+\frac{\beta s(p-q)}{2p}\int_{\Omega_r}|u|^p\mathrm{d}x,
\end{eqnarray*}
where we have used $\beta>0$.  Recall that  $\Omega_r$ is starshaped with respect to 0, so $x\cdot\textbf{n}\ge0$ for any $x\in\partial\Omega_r$. Thereby, applying the Gagliardo-Nirenberg inequality, we obtain
\begin{eqnarray*}
&&\frac{q-2}{2}m_{r,s}(\alpha)+\frac{\beta s(q-p)C_{p}}{2p}\alpha^{\frac{6-p}{4}}
\bigg(\int_{\Omega}|\nabla u|^2\mathrm{d}x\bigg)^{\frac{3p-6}{4}}\\
&\ge&\frac{q-2}{2}m_{r,s}(\alpha)+\frac{\beta s(q-p)}{2p}\int_{\Omega_r}|u|^p\mathrm{d}x\\
&=&\frac{q-2}{2}\bigg((a+b\int_{\Omega_r}|\nabla u|^2\mathrm{d}x)\int_{\Omega_r}|\nabla u|^2\mathrm{d}x
+\frac{1}{2}\int_{\Omega_r}V|u|^2\mathrm{d}x\bigg)\\
&&\quad-\frac{(a+b\int_{\Omega_r}|\nabla u|^2\mathrm{d}x)}{3}\int_{\Omega_r}|\nabla u|^2\mathrm{d}x-\frac{1}{6}\int_{\partial\Omega_r}|\nabla u|^2(x\cdot\textbf{n})\mathrm{d}\sigma-\frac{1}{6}\int_{\Omega_r}(\nabla V\cdot x)u^2\mathrm{d}x\\
&\ge&\frac{3q-10}{4}b\bigg(\int_{\Omega_r}|\nabla u|^2\mathrm{d}x\bigg)^2-\alpha\bigg(\frac{1}{6}\|\nabla V\cdot x\|_\infty+\frac{1}{N-2}\|V\|_\infty\bigg).
\end{eqnarray*}
As a consequence of $2<p<\frac{14}{3}$ and  Lemma \ref{mp-betage0}, we could bound $\int_{\Omega_r}|\nabla u|^2\mathrm{d}x$ uniformly in $s$ and $r$.
\end{proof}

\begin{lemma}\label{beta>0-e-strong-ge}
Assume that $0<\alpha<\tilde{\alpha}_V$ and  $r>\tilde{r}_\alpha$, where $\tilde{\alpha}_V$ and $\tilde{r}_\alpha$ are given in Theorem \ref{beta>0-e>0-Omega} and  Lemma \ref{mp-betage0}, respectively. Then the following hold:
 \begin{description}
   \item[(i)] Equation \eqref{main-eq-omega} admits a solution $(\lambda_{r,\alpha},u_{r,\alpha})$ such that $u_{r,\alpha}>0$ in $\Omega_r$.
   \item[(ii)] There exists $\bar{\alpha}\in\left(0,\tilde{\alpha}_V\right)$ such that
\[
\underset{r\to\infty}{\liminf}\,\lambda_{r,\alpha}>0\quad\text{for}\ \ 0<\alpha<\bar{\alpha}.
\]
 \end{description}
\end{lemma}
\begin{proof}
 (i) For fixed $0<\alpha<\tilde{\alpha}_V$ and  $r>\tilde{r}_\alpha$. Based on the preceding Theorem \ref{beta>0-s-omega-ge} and Lemma \ref{nabla-u-bdd-j}, we demonstrate that there exist solutions $\{(\lambda_{r,\alpha,s},u_{r,\alpha,s})\}$ to problem \eqref{main-eq-s-omega-2} for $s\in\left[1/2,1\right]$, and $\{u_{r,\alpha,s}\}\subset S_{r,\alpha}$ is bounded. Now, using a technique similar to the proof of  Theorem \ref{beta>0-s-omega-ge}, we obtain $u_{r,\alpha}\in S_{r,\alpha}$ and $\lambda_{r,\alpha}\in\mathbb{R}$ such that going to a subsequence, $\lambda_{r,\alpha,s}\to\lambda_{r,\alpha}$ and $u_{r,\alpha,s}\to u_{r,\alpha}$ in $H_0^{1}(\Omega_r)$ as $s\to1$. This combined with the strong maximum principle leads to that $u_{r,\alpha}>0$ is a solution of problem \eqref{main-eq-omega}.

(ii) In view of Lemma \ref{ur-ubdd} we have
\[
\underset{r\to\infty}{\limsup}\,\underset{\Omega_r}{\max}\,u_{r,\alpha}<\infty.
\]
Furthermore, let us \textbf{claim} that there exists $\alpha'_1>0$ such that
\begin{equation}\label{g-alpha>0}
g(\alpha):=\underset{r\to\infty}{\liminf}\,\underset{\Omega_r}{\max}\,u_{r,\alpha}>0
\end{equation}
for any $0<\alpha<\alpha'_1$. Assume
by  contradiction that there exists a sequence $\alpha_k\to0$  as $k\to\infty$ such that
$g(\alpha_k)=0$ for any $k$, that is,
\begin{equation*}
\underset{r\to\infty}{\limsup}\,\underset{\Omega_r}{\max}\,u_{r,\alpha_k}=0\quad\text{for any}\ k.
\end{equation*}
For any fixed $k$, it follows from  $u_{r,\alpha_k}\in S_{r,\alpha_k}$ that for any $t>2$
\begin{equation}\label{urk-m}
\int_{\Omega_r}|u_{r,\alpha_k}|^t\mathrm{d}x\le\big|\underset{\Omega_r}{\max}\,
u_{r,\alpha_k}\big|^{t-2}\alpha_k\to0\quad \text{as}\ r\to\infty.
\end{equation}
Hence, there exists $\bar{r}_k>0$ such that
\[
\Big|\frac{1}{q}\int_{\Omega_r}|u_{r,\alpha_k}|^{q}\mathrm{d}x
+\frac{\beta}{p}\int_{\Omega_r}|u_{r,\alpha_k}|^{p}\mathrm{d}x\Big|
<\frac{m_{r,1}(\alpha_k)}{2}\quad\text{for any}\ r\ge\bar{r}_k.
\]
In view of $E_r(u_{r,\alpha_k})=m_{r,1}(\alpha_k)$, we further have that for any large $k$
\begin{equation}\label{nabia-urk-s}
(a+b\int_{\Omega_r}|\nabla u_{r,\alpha_k}|^2\mathrm{d}x)\int_{\Omega_r}|\nabla u_{r,\alpha_k}|^2\mathrm{d}x+\int_{\Omega_r}V u_{r,\alpha_k}^2\mathrm{d}x
\ge m_{r,1}(\alpha_k),\ r\ge\bar{r}_k.
\end{equation}
It follows from \eqref{urk-m}-\eqref{nabia-urk-s} that there exists $r_k\ge \bar{r}_k$ with $\bar{r}_k\to\infty$ as $k\to\infty$ such that
\begin{equation}\label{urk-max}
\underset{k\to\infty}{\limsup}\,\underset{\Omega_{r_k}}{\max}\,u_{r_k,\alpha_k}=0,
\end{equation}
\begin{equation}\label{urk-s-2}
\int_{\Omega_{r_k}}|u_{r_k,\alpha_k}|^t\mathrm{d}x\le\big|\underset{\Omega_{r_k}}{\max}\,
u_{r_k,\alpha_k}\big|^{t-2}\alpha_k\to0\quad \text{as}\ k\to\infty\ \text{for any }t>2.
\end{equation}
and
\begin{equation}\label{nabia-urk-s-2}
(a+b\int_{\Omega_{r_k}}|\nabla u_{r,\alpha_k}|^2\mathrm{d}x)\int_{\Omega_{r_k}}|\nabla u_{r_k,\alpha_k}|^2\mathrm{d}x+\int_{\Omega_{r_k}}V u_{r_k,\alpha_k}^2\mathrm{d}x\to\infty\quad\text{as}\ k\to\infty.
\end{equation}
By Equation \eqref{main-eq-omega} and \eqref{urk-s-2}-\eqref{nabia-urk-s-2}
\begin{equation}\label{lambda-k}
\lambda_{r_k,\alpha_k}\to-\infty\ \text{as}\ k\to\infty.
\end{equation}
Now this together with \eqref{main-eq-omega} implies that
\begin{eqnarray*}
&&-(a+b\int_{\Omega_{r_k}}|\nabla u_{r,\alpha_k}|^2\mathrm{d}x)\Delta u_{r_k,\alpha_k}+\bigg(\|V\|_\infty
+\frac{\lambda_{r_k,\alpha_k}}{2}u_{r_k,\alpha_k}\bigg)u_{r_k,\alpha_k}\\
&\ge&\bigg(-\frac{\lambda_{r_k,\alpha_k}}{2}+|u_{r_k,\alpha_k}|^{q-2}
+\beta|u_{r_k,\alpha_k}|^{p-2}\bigg)u_{r_k,\alpha_k}\ge0
\end{eqnarray*}
for large $k$. Let $\theta_{r_k}$ be the principal eigenvalue of $-\Delta$ with Dirichlet boundary condition in $\Omega_{r_k}$, and $0<v_{r_k}$ be the corresponding normalized eigenfunction. Then $\theta_{r_k}=\theta_1/r_k^2$ and
\[
\bigg((a+b\int_{\Omega_{r_k}}|\nabla u_{r,\alpha_k}|^2\mathrm{d}x)\frac{\theta_1}{r_k^2}+\|V\|_\infty+\frac{\lambda_{r_k,\alpha_k}}{2}\bigg)
\int_{\Omega_{r_k}}u_{r_k,\alpha_k}v_{r_k}\mathrm{d}x\ge0.
\]
Since $\int_{\Omega_{r_k}}u_{r_k,\alpha_k}v_{r_k}\mathrm{d}x>0$, we have
\[
(a+b\int_{\Omega_{r_k}}|\nabla u_{r,\alpha_k}|^2\mathrm{d}x)\frac{\theta_1}{r_k^2}+\|V\|_\infty+\frac{\lambda_{r_k,\alpha_k}}{2}\ge0,
\]
which contradicts \eqref{lambda-k} for large $k$. Hence the \textbf{claim} holds.  Consider $H_0^1(\Omega_r)$ as a subspace of $H^1(\mathbb{R}^{3})$ for any $r>0$. In view of  Lemma \ref{nabla-u-bdd-j}, there are  $u_\alpha\in H^1(\mathbb{R}^{3})$ and $\lambda_\alpha\in\mathbb{R}$ such that as $r\to\infty$, up to a subsequence if necessary,
\begin{gather*}
  \|\nabla u_r\|_2\to\Lambda,\quad\lambda_{r,\alpha}\to\lambda_\alpha \quad\text{and}\quad
  u_{r,\alpha}\rightharpoonup u_\alpha\quad\text{in}\ H^1(\mathbb{R}^{3}),\\
  u_{r,\alpha}\to u_\alpha\quad\text{in}\ L_{loc}^{t}(\mathbb{R}^{3})\ \text{for all}\ 2\le t<6.
\end{gather*}
Assume by contradiction that $\lambda_{\alpha_n}\le0$ for some sequence $\alpha_n\to0$. Indeed, let us select
$\theta_r$ as the principal eigenvalue of $-\Delta$ with Dirichlet boundary condition in $\Omega_r$, the corresponding normalized eigenfunction $v_r>0$ can be chosen as a test function for  \eqref{main-eq-omega}, so there holds
\begin{eqnarray*}
&&\big((a+b\int_{\Omega_{r}}|\nabla u_{r,\alpha_n}|^2\mathrm{d}x)\theta_r+\lambda_{r,\alpha_n}\big)\int_{\Omega_r}u_{r,\alpha_n}v_r\mathrm{d}x
+\int_{\Omega_r}Vu_{r,\alpha_n}v_r\mathrm{d}x\\
&=&\beta\int_{\Omega_r}u^{p-1}_{r,\alpha_n}v_r\mathrm{d}x
+\int_{\Omega_r}u^{q-1}_{r,\alpha_n}v_r\mathrm{d}x\ge0.
\end{eqnarray*}
In view of $\int_{\Omega_r}u_{r,\alpha_n}v_r\mathrm{d}x>0$ and $\theta_r=r^{-2}\theta_1$, we get
\[
\lambda_{r,\alpha_n}+\underset{x\in\mathbb{R}^{3}}{\max}\,V(x)+r^{-2}\theta_1(a+b\int_{\Omega_{r}}|\nabla u_{r,\alpha_n}|^2\mathrm{d}x)\ge0.
\]
Hence there exists $C>0$ independent of $n$ such that $|\lambda_{r,\alpha_n}|\le C$ for any $n$.
Next, it is explained in two cases.

\textbf{Case 1:} There is subsequence, still denoted by $\{\alpha_n\}$, such that $u_{\alpha_n}=0$. We first \textbf{claim} that for every $n$ there exists $d_n>0$ such that
\begin{equation}\label{u-ge-dn}
\underset{r\to\infty}{\liminf}\underset{z\in\mathbb{R}^{3}}{\sup}\int_{B(z,1)}
u^2_{r,\alpha_n}\mathrm{d}x\ge d_n.
\end{equation}
Otherwise, the concentration compactness principle implies for every $n$ that
\[
u_{r,\alpha_n}\to0\ \text{in}\ L^t(\mathbb{R}^{3})\ \ \text{as}\ r\to\infty,\quad\text{for all }2<t\le6.
\]
By the diagonal principle, equation \eqref{main-eq-omega} and $|\lambda_{r,\alpha_n}|\le C$ for large $r$, there exists $r_n\to\infty$ such that
\[
\int_{\Omega_{r_n}}|\nabla u_{r_n,\alpha_n}|^2\mathrm{d}x\to0,
\]
which is contradicts (iii) in Lemma \ref{mp-betage0} for large $n$. As a consequence \eqref{u-ge-dn} holds, and then there is $x_{r,\alpha_n}\in\Omega_r$ with $|x_{r,\alpha_n}|\to\infty$ such that
\[
\int_{B(x_{r,\alpha_n},1)}
u^2_{r,\alpha_n}\mathrm{d}x\ge\frac{d_n}{2}.
\]
Next, we \textbf{claim} that
\begin{equation}\label{dist}
\text{dist}(x_{r,\alpha_n},\partial\Omega_r)\to\infty,\quad \text{as}\ r\to\infty.
\end{equation}
 Indeed, assume by contradiction that
$\underset{r\to\infty}{\liminf}\,\text{dist}(x_{r,\alpha_n},\partial\Omega_r)=l<\infty$.
 It follows from \eqref{g-alpha>0} that $l>0$.
Let $v_r(x)=u_{r,\alpha_n}(x+x_{r,\alpha_n})$ for any $x\in\Sigma_r:=\{x\in\mathbb{R}^{3}:x+x_{r,\alpha_n}\in\Omega_r\}$. Then $v_r$ is bounded in $H^1(\mathbb{R}^{3})$, and there is $v\in H^1(\mathbb{R}^{3})$ such that $v_r\rightharpoonup v$ as $r\to\infty$. By the regularity theory of elliptic partial differential equations and $\liminf_{r\to\infty}u_{r,\alpha_n}(x_{r,\alpha_n})\ge d_{\alpha_n}> 0$, we deduce that $v(0)\ge d_{\alpha_n}>0$. Assume without loss of generality that, up to a subsequence,
\[
\underset{r\to\infty}{\lim}\frac{x_{r,\alpha_n}}{|x_{r,\alpha_n}|}=e_1.
\]
Setting
\[
\Sigma=\Big\{x\in\mathbb{R}^{3}:x\cdot e_1<l\Big\}=\Big\{x\in\mathbb{R}^{3}:x_1<l\Big\},
\]
we have $\varphi(\cdot-x_{r,\alpha_n})\in C_c^\infty(\Omega_r)$ for any $\varphi\in C_c^\infty(\Omega_r)$ and $r$ large enough. It then follows that
\begin{eqnarray}\label{eq-varphi}
&&\int_{\Omega_r}|u_{r,\alpha_n}|^{q-2}u_{r,\alpha_n}\varphi(\cdot-x_{r,\alpha_n})\mathrm{d}x
+\beta\int_{\Omega_r}|u_{r,\alpha_n}|^{p-2}u_{r,\alpha_n}\varphi(\cdot-x_{r,\alpha_n})\mathrm{d}x\notag\\
&=&(a+b\int_{\Omega_r}|\nabla u_{r,\alpha_n}|^2\mathrm{d}x)\int_{\Omega_r}\nabla u_{r,\alpha_n}\nabla\varphi(\cdot-x_{r,\alpha_n})\mathrm{d}x
+\int_{\Omega_r}Vu_{r,\alpha_n}\varphi(\cdot-x_{r,\alpha_n})\mathrm{d}x\notag\\
&&\quad+\lambda_{r,\alpha_n}\int_{\Omega_r}u_{r,\alpha_n}\varphi(\cdot-x_{r,\alpha_n})\mathrm{d}x
.
\end{eqnarray}
Since $|x_{r,\alpha_n}|\to\infty$ as $r\to\infty$, we have
\begin{eqnarray}\label{V-vanish}
\Big|\int_{\Omega_r}Vu_{r,\alpha_n}\varphi(\cdot-x_{r,\alpha_n})\mathrm{d}x\Big|
&\le&\int_{\text{Supp}\varphi}\Big|V(\cdot+x_{r,\alpha_n})v_{r}\varphi\Big|\mathrm{d}x\notag\\
&\le&\|v_{r}\|_{6}\|\varphi\|_{6}
\bigg(\int_{\text{Supp}\varphi}|V(\cdot+x_{r,\alpha_n})|^{\frac{3}{2}}\mathrm{d}x\bigg)^{\frac{2}{3}}\\
&\le&\|v_{r}\|_{6}\|\varphi\|_{6}
\bigg(\int_{\mathbb{R}^{3}\setminus B_{\frac{|x_{r,\alpha_n}|}{2}}}|V(\cdot+x_{r,\alpha_n})|
^{\frac{3}{2}}\mathrm{d}x\bigg)^{\frac{2}{3}}\to0\quad\text{as}\ r\to\infty\notag.
\end{eqnarray}
Letting $r\to\infty$ in \eqref{eq-varphi}, we obtain for $\varphi\in C_c^\infty(\Sigma)$:
\[
(a+b\Lambda^2)\int_{\Sigma}\nabla v\cdot\nabla\varphi\mathrm{d}x
+\lambda_{\alpha_n}\int_{\Sigma}v\varphi\mathrm{d}x
=\int_{\Sigma}|v|^{q-2}v\varphi\mathrm{d}x
+\beta\int_{\Sigma}|v|^{p-2}v\varphi\mathrm{d}x.
\]
Thus $v\in H^1(\mathbb{R}^{3})$ is a weak solution of the equation
\begin{equation}\label{eq-w-alpha-ge}
-(a+b\Lambda^2)\Delta v+\lambda_{\alpha_n}v=|v|^{q-2}v+\beta|v|^{p-2}v\quad\text{in}\ \Sigma.
\end{equation}
Hence we obtain a nontrivial nonnegative solution of \eqref{eq-w-alpha-ge} on a half space, which
is impossible (see e.g., \cite{EL82}). This proves that dist$(x_{r,\alpha_n},\partial\Omega_r)\to\infty$ as $r\to\infty$. A similar argument as above shows that \eqref{eq-w-alpha-ge} holds for $\Sigma=\mathbb{R}^{3}$.
Thus
\begin{equation}\label{alpha-eq-j-ge}
(a+b\Lambda^2)\int_{\mathbb{R}^{3}}|\nabla v|^2\mathrm{d}x
+\lambda_{\alpha_n}\int_{\mathbb{R}^{3}}|v|^2\mathrm{d}x
=\int_{\mathbb{R}^{3}}|v|^{q}\mathrm{d}x
+\beta\int_{\mathbb{R}^{3}}|v|^{p}\mathrm{d}x.
\end{equation}
The Pohozaev identity gives
\begin{equation}\label{alpha-eq-pohozaev-j-ge}
\frac{(a+b\Lambda^2)}{6}\int_{\mathbb{R}^{3}}|\nabla v|^2\mathrm{d}x
+\frac{\lambda_{\alpha_n}}{2}\int_{\mathbb{R}^{3}}v^2\mathrm{d}x
=\frac{1}{q}\int_{\mathbb{R}^{3}}|v|^{q}\mathrm{d}x
+\frac{\beta}{p}\int_{\mathbb{R}^{3}}|v|^p\mathrm{d}x.
\end{equation}
Next it follows from \eqref{alpha-eq-j-ge}-\eqref{alpha-eq-pohozaev-j-ge} and $2<p<q<6$ that
\begin{equation}\label{lambdaalpha-p-ge}
\frac{1}{3}\lambda_{\alpha_n}\int_{\mathbb{R}^{3}}v^2\mathrm{d}x
=\frac{\beta(6-p)}{6p}\int_{\mathbb{R}^{3}}|v|^p\mathrm{d}x
+\frac{6-q}{6q}\int_{\mathbb{R}^{3}}|v|^{q}\mathrm{d}x.
\end{equation}
As a result, we have $\lambda_{\alpha_n}>0$ which is a contradiction.

\textbf{Case 2:} $u_{\alpha_n}\neq0$ for $n$ large. Note that $u_{\alpha_n}$ satisfies
\begin{equation}\label{eq-w-alpha-j-u}
-(a+b\Lambda^2)\Delta u_{\alpha_n}+\lambda_{\alpha_n}u_{\alpha_n}=|u_{\alpha_n}|^{q-2}u_{\alpha_n}
+\beta|u_{\alpha_n}|^{p-2}u_{\alpha_n}\quad\text{in}\ \mathbb{R}^{3}.
\end{equation}
Set $w_{r,\alpha_n}:=u_{r,\alpha_n}-u_{\alpha_n}$, we first \textbf{claim} that  there are $d_n>0$ and $z_{r,\alpha_n}\in\Omega_r$ with $|z_{r,\alpha_n}|\to\infty$ as $r\to\infty$ such that
\begin{equation}\label{wdn}
\int_{B(z_{r,\alpha_n},1)}w^2_{r,\alpha_n}\mathrm{d}x>d_n.
\end{equation}
If not,
\[
\underset{r\to\infty}{\liminf}\underset{z\in\mathbb{R}^{3}}{\sup}
\int_{B(z,1)}w^2_{r,\alpha_n}\mathrm{d}x=0,
\]
then the concentration compactness principle implies $u_{r,\alpha_n}\to u_{\alpha_n}$ in $L^t(\mathbb{R}^{3})$ for any $2<t<6$. It follows from  \eqref{main-eq-omega}  and \eqref{eq-w-alpha-j-u} that
\begin{eqnarray*}
&&(a+b\int_{\Omega_r}|\nabla u_{r,\alpha_n}|^2\mathrm{d}x)\int_{\Omega_r}|\nabla u_{r,\alpha_n}|^2\mathrm{d}x+\alpha_n\lambda_{r,\alpha_n}\\
&=&\beta\int_{\Omega_r}|u_{r,\alpha_n}|^p\mathrm{d}x
+\int_{\Omega_r}|u_{r,\alpha_n}|^{q}\mathrm{d}x
-\int_{\Omega_r}Vu_{r,\alpha_n}\mathrm{d}x\\
&=&\beta\int_{\mathbb{R}^{3}}|u_{r,\alpha_n}|^p\mathrm{d}x
+\int_{\mathbb{R}^{3}}|u_{r,\alpha_n}|^{q}\mathrm{d}x
-\int_{\mathbb{R}^{3}}Vu_{r,\alpha_n}\mathrm{d}x\\
&\to&\beta\int_{\mathbb{R}^{3}}|u_{\alpha_n}|^p\mathrm{d}x
+\int_{\mathbb{R}^{3}}|u_{\alpha_n}|^{q}\mathrm{d}x
-\int_{\mathbb{R}^{3}}Vu_{\alpha_n}\mathrm{d}x\\
&=&(a+b\Lambda^2)\int_{\mathbb{R}^{3}}|\nabla u_{\alpha_n}|^2\mathrm{d}x
+\lambda_{\alpha_n}\int_{\mathbb{R}^{3}}|u_{\alpha_n}|^2\mathrm{d}x.
\end{eqnarray*}
Recall that  $\lambda_{r,\alpha_n}\to\lambda_{\alpha_n}$ as $r\to\infty$, there holds
\begin{eqnarray*}
&&(a+b\int_{\Omega_r}|\nabla u_{r,\alpha_n}|^2\mathrm{d}x)\int_{\Omega_r}|\nabla u_{r,\alpha_n}|^2\mathrm{d}x+\lambda_{r,\alpha_n}\alpha_n\\
&\to&(a+b\Lambda^2)\int_{\mathbb{R}^{3}}|\nabla u_{\alpha_n}|^2\mathrm{d}x
+\lambda_{\alpha_n}\int_{\mathbb{R}^{3}}|u_{\alpha_n}|^2\mathrm{d}x\quad\text{as}\ r\to\infty.
\end{eqnarray*}
This together with (iii) in Lemma \ref{mp-betage0} and $|\lambda_{\alpha_n}|\le C$ for large $n$, we have
\[
(a+b\Lambda^2)\int_{\mathbb{R}^{3}}|\nabla u_{\alpha_n}|^2\mathrm{d}x\to\infty \quad\text{as}\ n\to\infty.
\]
By \eqref{eq-w-alpha-j-u} and the Pohozaev identity,
\begin{eqnarray*}
0\le\frac{2-q}{2q}\lambda_{\alpha_n}\int_{\mathbb{R}^{3}}u_{\alpha_n}^2\mathrm{d}x
&=&\frac{q-6}{6q}(a+b\Lambda^2)\int_{\mathbb{R}^{3}}|\nabla u_{\alpha_n}|^2\mathrm{d}x
+\frac{1}{6}\int_{\mathbb{R}^{3}}\tilde{V}u_{\alpha_n}^2\mathrm{d}x\\
&&\quad+\frac{q-2}{2q}\int_{\mathbb{R}^{3}}Vu_{\alpha_n}^2\mathrm{d}x
-\frac{(q-p)\beta}{pq}\int_{\mathbb{R}^{3}}|u_{\alpha_n}|^p\mathrm{d}x\notag\\
&\le&\frac{q-6}{6q}(a+b\Lambda^2)\int_{\mathbb{R}^{3}}|\nabla u_{\alpha_n}|^2\mathrm{d}x
+\frac{\|\tilde{V}\|_\infty}{6}\alpha_n+\frac{q-2}{2q}\|V\|_\infty\alpha_n\\
&\to&-\infty\quad\text{as}\ n\to\infty,
\end{eqnarray*}
which is impossible.
Then \eqref{wdn} hold, that is, $\tilde{w}_{r,\alpha_n}:=w_{r,\alpha_n}(\cdot+z_{r,\alpha_n})\rightharpoonup\tilde{w}_{\alpha_n}\neq0$,
and a calculation similar to \eqref{dist} and \eqref{V-vanish}, we have $\tilde{w}_{\alpha_n}$ is a nonnegative solution of
\[
-(a+b\Lambda^2)\Delta \tilde{w}_{\alpha_n}+\lambda_{\alpha_n}\tilde{w}_{\alpha_n}
=|\tilde{w}_{\alpha_n}|^{q-2}\tilde{w}_{\alpha_n}
+\beta|\tilde{w}_{\alpha_n}|^{p-2}\tilde{w}_{\alpha_n}\quad\text{in}\ \mathbb{R}^{3}.
\]
It follows from an argument similar to  \eqref{lambdaalpha-p-ge} that $\lambda_n>0$ for large $n$, which is a contradiction. The proof is now complete.
\end{proof}

\noindent\textbf{Proof of Theorem \ref{beta>0-e>0-Omega}:} The proof is an immediate consequence of Theorem \ref{beta>0-s-omega-ge} and Lemma \ref{beta>0-e-strong-ge}.\qed

\section{Proof of Theorem \ref{betale0-e>0-Omega}}
\setcounter{equation}{0}
\setcounter{theorem}{0}	
In this section, we always assume that the assumptions of Theorem \ref{betale0-e>0-Omega} hold. In
order to obtain a bounded Palais-Smale sequence, we will use the monotonicity trick.
For $\frac{1}{2}\le s\le1$, the energy functional $E_{r,s}:S_{r,\alpha}\to\mathbb{R}$ is defined by
\[
E_{r,s}(u)=\frac{a}{2}\int_{\Omega_r}|\nabla u|^2dx+\frac{b}{4}\bigg(\int_{\Omega_r}|\nabla u|^2dx\bigg)^2+\frac{1}{2}\int_{\Omega_r}V(x)u^2dx
-\frac{s}{q}\int_{\Omega_r}|u|^{q}dx-\frac{\beta}{p}\int_{\Omega_r}|u|^pdx.
\]
Note that  if $u\in S_{r,\alpha}$ is a critical point of $E_{r,\alpha}$, then there exists $\lambda\in\mathbb{R}$ such that $(\lambda,u)$ is a solution of the equation
\begin{equation}\label{main-eq-s-omega}
\begin{cases}
-(a+b\int_{\Omega_r}|\nabla u|^2dx)\Delta u+V(x)u+\lambda u=s|u|^{q-2}u+\beta|u|^{p-2}u&\text{in}\ \Omega_r,\\
u\in H^{1}_{0}(\Omega_r),\ \int_{\Omega_r}|u|^2dx=\alpha.
\end{cases}
\end{equation}

\begin{lemma}\label{betale0-mp-g}
For any $\alpha>0$, there exists $r_\alpha>0$ and $u^0,u^1\in S_{r,\alpha}$ such that
\begin{description}
  \item[(i)] $E_{r,s}(u^1)\le0$ for any $r>r_\alpha$ and $s\in\left[\frac{1}{2},1\right]$,
  \[
  \|\nabla u^0\|_2^2<\bigg(\frac{2q}{3C_q(q-2)}
(a-\|V_-\|_{\frac{3}{2}}S^{-1})\alpha^{\frac{q-6}{4}}\bigg)^{\frac{4}{3q-10}}<\|\nabla u^1\|_2^2
  \]
  and
  \[
  E_{r,s}(u^0)<\frac{(3q-10)(a-\|V_-\|_{\frac{3}{2}}S^{-1})}{6(q-2)}
\bigg(\frac{2q(a-\|V_-\|_{\frac{3}{2}}S^{-1})\alpha^{\frac{q-6}{4}}}{3C_q(q-2)}
\bigg)^{\frac{4}{3q-10}}
  \]
  \item[(ii)] If $u\in S_{r,\alpha}$ satisfies
  \[
  \|\nabla u\|_2^2=\bigg(\frac{2q}{3C_q(q-2)}
(a-\|V_-\|_{\frac{3}{2}}S^{-1})\alpha^{\frac{q-6}{4}}\bigg)^{\frac{4}{3q-10}}
  \]
  then there holds
  \[
  E_{r,s}(u)\ge\frac{(3q-10)(a-\|V_-\|_{\frac{3}{2}}S^{-1})}{6(q-2)}
\bigg(\frac{2q(a-\|V_-\|_{\frac{3}{2}}S^{-1})\alpha^{\frac{q-6}{4}}}{3C_q(q-2)}
\bigg)^{\frac{4}{3q-10}}.
  \]
  \item[(iii)] Set
  \[
  m_{r,s}(\alpha)=\underset{\gamma\in\Gamma_{r,\alpha}}{\inf}\,
  \underset{t\in\left[0,1\right]}{\sup}\,E_{r,s}(\gamma(t)),
  \]
  with
  \[
  \Gamma_{r,\alpha}=\Big\{\gamma\in C(\left[0,1\right],S_{r,\alpha}):\gamma(0)=u^0,\gamma(1)=u^1\Big\}.
  \]
  Then
  \[
  \frac{(3q-10)(a-\|V_-\|_{\frac{3}{2}}S^{-1})}{6(q-2)}
\bigg(\frac{2q(a-\|V_-\|_{\frac{3}{2}}S^{-1})\alpha^{\frac{q-6}{4}}}{3C_q(q-2)}
\bigg)^{\frac{4}{3q-10}}\le m_{r,s}(\alpha)<H_\alpha.
  \]
  Here $H_\alpha:=\underset{t\in\mathbb{R}^+}{\max}\,h(t)$, where
  \[
  h(t)=\frac{1}{2}\bigg(a+\|V\|_{\frac{3}{2}}S^{-1}\bigg)t^2\theta\alpha+\frac{b}{4}t^4\theta^2\alpha^2
-\frac{\beta C_p}{p}\alpha^{\frac{p}{2}}\theta^{\frac{3(p-2)}{4}}t^{\frac{3(p-2)}{2}}-\frac{1}{2q}t^{\frac{3(q-2)}{2}}\alpha^{\frac{q}{2}}|\Omega|^{\frac{2-q}{2}}.
  \]
\end{description}
\end{lemma}
\begin{proof}
(i): Clearly the set $S_{r,\alpha}$ is path connected.  Let  $v_t(x)=t^{\frac{3}{2}}v_1(tx)$, where $v_1\in S_{1,\alpha}$ is the positive normalized eigenfunction of $-\Delta$ with Dirichlet boundary condition in $\Omega$ associated to $\theta$.  Owing to
\[
\int_{\Omega}|\nabla v_1|^2\mathrm{d}x=\theta\alpha \quad\text{and}\quad\int_{\Omega}|v_1|^q\mathrm{d}x\ge\alpha^{\frac{q}{2}}|\Omega|^{\frac{2-q}{2}},
\]
for $x\in\Omega_{\frac{1}{t}}$ and $t>0$ there holds
\begin{eqnarray}\label{e-vt-s-h}
E_{\frac{1}{t},s}(v_t)&\le&\frac{1}{2}t^2\bigg(a+\|V\|_{\frac{3}{2}}S^{-1}\bigg)\int_{\Omega}|\nabla v_1|^2\mathrm{d}x
+\frac{b}{4}t^4\bigg(\int_{\Omega}|\nabla v_1|^2\mathrm{d}x\bigg)^2\notag\\
&&\quad-\frac{\beta}{p}t^{\frac{3(p-2)}{2}}\int_{\Omega}|v_1|^p\mathrm{d}x
-\frac{1}{2q}t^{\frac{3(q-2)}{2}}\int_{\Omega}|v_1|^{q}\mathrm{d}x\\
&\le& h(t)\notag.
\end{eqnarray}
Note that since $\frac{14}{3}<q<6$ and $\beta<0$, there exist $0<T_\alpha<t_\alpha$ such that $h(t_\alpha)=0$, $h(t)<0$ for any $t>t_\alpha$, $h(t)>0$ for any $0<t<t_\alpha$ and $h(T_\alpha)=\underset{t\in\mathbb{R}^+}{\max}\,h(t)$. As a consequence, there holds
\begin{equation}\label{htalpha0}
E_{r,s}(v_{t_\alpha})=E_{\frac{1}{t_{\alpha}},s}(v_{t_\alpha})\le h(t_\alpha)=0
\end{equation}
for any $r\ge\frac{1}{t_\alpha}$ and $s\in\left[\frac{1}{2},1\right]$. Moreover, there exists $0<t_1<T_\alpha$ such that for any $t\in\left[0,t_1\right)$,
\begin{equation}\label{htupbdd}
h(t)<h(t_1)\le\frac{(3q-10)(a-\|V_-\|_{\frac{3}{2}}S^{-1})}{6(q-2)}
\bigg(\frac{2q(a-\|V_-\|_{\frac{3}{2}}S^{-1})\alpha^{\frac{q-6}{4}}}{3C_q(q-2)}
\bigg)^{\frac{4}{3q-10}}.
\end{equation}

On the other hand, it follows from the Gagliardo-Nirenberg inequality and the H\"{o}lder inequality that
\begin{equation}\label{ers-bel}
E_{r,s}(u)\ge\frac{a}{2}\bigg(1-\|V_-\|_{\frac{3}{2}}S^{-1}\bigg)\int_{\Omega_r}|\nabla u|^2dx-\frac{C_q\alpha^{\frac{6-q}{4}}}{q}\bigg(\int_{\Omega_r}|\nabla u|^2dx\bigg)^{\frac{3(q-2)}{4}}.
\end{equation}
Let
\[
f(t):=\frac{1}{2}\bigg(a-\|V_-\|_{\frac{3}{2}}S^{-1}\bigg)t
-\frac{C_q\alpha^{\frac{6-q}{4}}}{q}t^{\frac{3(q-2)}{4}}
\]
and
\[
\tilde{t}:=\bigg(\frac{2q}{3C_q(q-2)}
(a-\|V_-\|_{\frac{3}{2}}S^{-1})\alpha^{\frac{q-6}{4}}\bigg)^{\frac{4}{3q-10}}.
\]
Then $f$ is increasing in $(0,\tilde{t})$ and decreasing in $(\tilde{t},\infty)$, and
\[
f(\tilde{t})=\frac{(3q-10)(a-\|V_-\|_{\frac{3}{2}}S^{-1})}{6(q-2)}
\bigg(\frac{2q(a-\|V_-\|_{\frac{3}{2}}S^{-1})\alpha^{\frac{q-6}{4}}}{3C_q(q-2)}
\bigg)^{\frac{4}{3q-10}}.
\]
For $r>\tilde{r}_\alpha:=
\max\Big\{\frac{1}{t_1},\sqrt{\frac{2\theta\alpha}{\tilde{t}}}\Big\}$, we have $v_{\frac{1}{\tilde{r}_\alpha}}\in S_{\tilde{r}_\alpha,\alpha}\subset S_{r,\alpha}$ and
\begin{equation}\label{nabla-u0}
\|\nabla v_{\frac{1}{\tilde{r}_\alpha}}\|_2^2=\bigg(\frac{1}{\tilde{r}_\alpha}\bigg)^2\|\nabla v_1\|_2^2<\bigg(\frac{2q}{3C_q(q-2)}
(a-\|V_-\|_{\frac{3}{2}}S^{-1})\alpha^{\frac{q-6}{4}}\bigg)^{\frac{4}{3q-10}}.
\end{equation}
Moreover, there holds
\begin{equation}\label{e-ralpha-s-ht1}
E_{\tilde{r}_\alpha,s}(v_{\frac{1}{\tilde{r}_\alpha}})\le h\bigg(\frac{1}{\tilde{r}_\alpha}\bigg)<h(t_1).
\end{equation}
Setting $u^0=v_{\frac{1}{\tilde{r}_\alpha}}$, $u^1=v_{t_\alpha}$ and
\[
r_\alpha=\max\Big\{\frac{1}{t_\alpha},\tilde{r}_\alpha\Big\},
\]
the statement (i) holds due to  \eqref{htalpha0}-\eqref{htupbdd} and \eqref{nabla-u0}-\eqref{e-ralpha-s-ht1}. Note that statement (ii) holds by \eqref{ers-bel} and a direct calculation.
 (iii): Since $E_{r,s}(u^1)\le0$ for any $\gamma\in\Gamma_{r,\alpha}$, we have
\[
\|\nabla \gamma(0)\|_2^2<\tilde{t}<\|\nabla \gamma(1)\|_2^2.
\]
It then follows from \eqref{ers-bel} that
\[
\underset{t\in\left[0,1\right]}{\max}E_{r,s}(\gamma(t))\ge f(\tilde{t})=\frac{(3q-10)(a-\|V_-\|_{\frac{3}{2}}S^{-1})}{6(q-2)}
\bigg(\frac{2q(a-\|V_-\|_{\frac{3}{2}}S^{-1})\alpha^{\frac{q-6}{4}}}{3C_q(q-2)}
\bigg)^{\frac{4}{3q-10}}
\]
for any $\gamma\in\Gamma_{r,\alpha}$, hence the first inequality  holds. For the second inequality,  we define a path $\gamma_0\in\Gamma_{r,\alpha}$ by
\[
\gamma_0(\tau):\Omega_r\to\mathbb{R},\quad x\mapsto\bigg(\tau t_\alpha+(1-\tau)\frac{1}{\tilde{r}_\alpha}\bigg)^{\frac{3}{2}}v_1\bigg(\bigg(\tau t_\alpha+(1-\tau)\frac{1}{\tilde{r}_\alpha}\bigg)x\bigg)
\]
for $\tau\in\left[0,1\right]$.
\end{proof}

In view of Lemma \ref{betale0-mp-g}, the energy functional $E_{r,s}$ possesses the mountain pass geometry.  For the sequence obtained from  Theorem \ref{mp-th}, we will show its convergence in the next theorem.
\begin{theorem}\label{beta>0-s-omega}
For $r>r_\alpha$, where $r_\alpha$ is defined in Lemma \ref{betale0-mp-g}. Problem \eqref{main-eq-s-omega} admits a solution $(\lambda_{r,s},u_{r,s})$ for almost every $s\in\left[\frac{1}{2},1\right]$. Moreover, $u_{r,s}\ge0$ and $E_{r,s}(u_{r,s})=m_{r,s}(\alpha)$.
\end{theorem}
\begin{proof}
Using Lemma \ref{betale0-mp-g} and  Theorem \ref{mp-th}, the proof is similar to that of Theorem \ref{beta>0-s-omega-ge}, so we omit here.
\end{proof}

Next, we will further establish a uniform estimate in order to obtain a solution of problem \eqref{main-eq-s-omega}.
\begin{lemma}\label{nabla-u-bdd}
Suppose that $(\lambda,u)\in\mathbb{R}\times S_{r,\alpha}$ is a solution of problem \eqref{main-eq-s-omega} established in Theorem \ref{beta>0-s-omega} for some $r$ and $s$, then
\[
\int_{\Omega_r}|\nabla u|^2\mathrm{d}x\le \frac{12}{a(3q-10)}\bigg(\frac{q-2}{2}H_\alpha+\alpha\bigg(\frac{1}{6}\|\nabla V\cdot x\|_\infty+\frac{q-2}{4}\|V\|_\infty\bigg)\bigg).
\]
\end{lemma}
\begin{proof}
 By an argument similar to that in Lemma \ref{nabla-u-bdd-j}, we deduce that
\begin{eqnarray*}
&&\frac{(a+b\int_{\Omega_r}|\nabla u|^2\mathrm{d}x)}{3}\int_{\Omega_r}|\nabla u|^2\mathrm{d}x-\frac{1}{6}\int_{\partial\Omega_r}|\nabla u|^2(x\cdot\textbf{n})\mathrm{d}\sigma-\frac{1}{6}\int_{\Omega_r}(\nabla V\cdot x)u^2\mathrm{d}x\\
&=&\frac{(q-2)s}{2q}\int_{\Omega_r}|u|^{q}\mathrm{d}x+\frac{\beta(p-2)}{2p}\int_{\Omega_r}|u|^p\mathrm{d}x\\
&\ge&\frac{q-2}{2}\bigg(\frac{s}{q}\int_{\Omega_r}|u|^{q}\mathrm{d}x
+\frac{\beta}{p}\int_{\Omega_r}|u|^p\mathrm{d}x\bigg)\\
&=&\frac{q-2}{2}\bigg(\frac{(a+b\int_{\Omega_r}|\nabla u|^2\mathrm{d}x)}{2}\int_{\Omega_r}|\nabla u|^2\mathrm{d}x
+\frac{1}{2}\int_{\Omega_r}V|u|^2\mathrm{d}x-m_{r,s(\alpha)}\bigg),
\end{eqnarray*}
where $\beta\le0$ have used. Recall that  $\Omega_r$ is starshaped with respect to 0, so $x\cdot\textbf{n}\ge0$ for any $x\in\partial\Omega_r$, thereby,
\begin{eqnarray*}
\frac{q-2}{2}m_{r,s}(\alpha)&\ge&\frac{q-2}{2}\bigg(\frac{(a+b\int_{\Omega_r}|\nabla u|^2\mathrm{d}x)}{2}\int_{\Omega_r}|\nabla u|^2\mathrm{d}x
+\frac{1}{2}\int_{\Omega_r}V|u|^2\mathrm{d}x\bigg)\\
&&\quad-\bigg(\frac{(a+b\int_{\Omega_r}|\nabla u|^2\mathrm{d}x)}{3}\int_{\Omega_r}|\nabla u|^2\mathrm{d}x
-\frac{1}{6}\int_{\partial\Omega_r}|\nabla u|^2(x\cdot\textbf{n})\mathrm{d}\sigma\\
&&\quad-\frac{1}{6}\int_{\Omega_r}(\nabla V\cdot x)u^2\mathrm{d}x\bigg)\\
&\ge&\frac{3q-10}{12}a\int_{\Omega_r}|\nabla u|^2\mathrm{d}x-\alpha\bigg(\frac{1}{6}\|\nabla V\cdot x\|_\infty+\frac{q-2}{4}\|V\|_\infty\bigg).
\end{eqnarray*}
According to  $m_{r,s}(\alpha)<H_\alpha$, the proof of Lemma \ref{nabla-u-bdd} is now complete.
\end{proof}

\begin{lemma}\label{lambda-r>0-betale0}

$(i)$For every $\alpha>0$,  problem \eqref{main-eq-omega} has a solution $(\lambda_r,u_r)$ provided $r>r_\alpha$ where $r_\alpha$ is as in Lemma \ref{betale0-mp-g}. Moreover, $u_r>0$ in $\Omega_r$.

 $(ii)$ Let $(\lambda_{r,\alpha},u_{r,\alpha})$ is the solution of problem \eqref{main-eq-omega} from Theorem \ref{beta>0-s-omega}. Then there exists $\tilde{\alpha}>0$ such that
\[
\underset{r\to\infty}{\liminf}\,\lambda_{r,\alpha}>0\quad\text{for}\quad0<\alpha<\tilde{\alpha}.
\]
\end{lemma}

\begin{proof}
 The proof of (i) is similar to that of Lemma \ref{beta>0-e-strong-ge}, so we omit it here.
For (ii), let $(\lambda_{r,\alpha},u_{r,\alpha})$ be the solution of problem \eqref{main-eq-omega}. The regularity theory of elliptic partial differential equations yields $u_{r,\alpha}\in C(\Omega_r)$.  By \eqref{g-alpha>0}, we similarly get that  there exists $\alpha'_2>0$ such that
\begin{equation}\label{g-alpha}
g(\alpha)=\underset{r\to\infty}{\liminf}\,\underset{\Omega_r}{\max}\,u_{r,\alpha}>0
\end{equation}
for any $0<\alpha<\alpha'_2$.
Consider $H_0^1(\Omega_r)$ as a subspace of $H^1(\mathbb{R}^{3})$ for any $r>0$. It follows from Lemma \ref{nabla-u-bdd} that the set of solutions $\{u_{r,\alpha}:r>r_\alpha\}$ established in Theorem \ref{beta>0-s-omega} is bounded in $H^1(\mathbb{R}^{3})$, we assume that $\int_{\Omega_r}|\nabla u_{r,\alpha}|^2\mathrm{d}x\to\Lambda^2$, so there exist $u_\alpha\in H^1(\mathbb{R}^{3})$ and $\lambda_\alpha\in\mathbb{R}$ such that as $r\to\infty$ up to a subsequence:
\begin{gather*}
  \lambda_{r,\alpha}\to\lambda_\alpha,
  \quad u_{r,\alpha}\rightharpoonup u_\alpha\quad\text{in}\ H^1(\mathbb{R}^{3}),\\
  u_{r,\alpha}\to u_\alpha\quad\text{in}\ L_{loc}^{t}(\mathbb{R}^{3})\ \text{for all}\ 2\le t<6,
\end{gather*}
and $u_\alpha$ is a solution of the equation
\[
-(a+b\Lambda^2)\Delta u_\alpha+Vu_\alpha+\lambda_\alpha u_\alpha=|u_\alpha|^{q-2}u_\alpha
+\beta|u_\alpha|^{p-2}u_\alpha\quad \text{in}\ \mathbb{R}^{3}.
\]
Next, it is explained in two cases.

\textbf{Case1:} $u_\alpha\neq0$ for $\alpha>0$ small.
Therefore,
\begin{equation}\label{alpha-eq}
(a+b\Lambda^2)\int_{\mathbb{R}^{3}}|\nabla u_\alpha|^2\mathrm{d}x+\int_{\mathbb{R}^{3}}Vu_\alpha^2\mathrm{d}x
+\lambda_\alpha\int_{\mathbb{R}^{3}}u_\alpha^2\mathrm{d}x
=\int_{\mathbb{R}^{3}}|u_\alpha|^{q}\mathrm{d}x
+\beta\int_{\mathbb{R}^{3}}|u_\alpha|^p\mathrm{d}x
\end{equation}
and the Pohozaev identity gives
\begin{eqnarray}\label{alpha-eq-pohozaev}
&&\frac{(a+b\Lambda^2)}{6}\int_{\mathbb{R}^{3}}|\nabla u_\alpha|^2\mathrm{d}x
+\frac{1}{6}\int_{\mathbb{R}^{3}}\tilde{V}u_\alpha^2\mathrm{d}x
+\frac{1}{2}\int_{\mathbb{R}^{3}}Vu_\alpha^2\mathrm{d}x
+\frac{\lambda_\alpha}{2}\int_{\mathbb{R}^{3}}u_\alpha^2\mathrm{d}x\notag\\
&=&\frac{1}{q}\int_{\mathbb{R}^{3}}|u_\alpha|^{q}\mathrm{d}x
+\frac{\beta}{p}\int_{\mathbb{R}^{3}}|u_\alpha|^p\mathrm{d}x.
\end{eqnarray}
Next it follows from \eqref{alpha-eq} and \eqref{alpha-eq-pohozaev} that
\begin{eqnarray*}
\frac{2a-\|\tilde{V}_+\|_{\frac{3}{2}}S^{-1}}{6}\int_{\mathbb{R}^{3}}|\nabla u_\alpha|^2\mathrm{d}x&\le&\frac{(a+b\Lambda^2)}{3}\int_{\mathbb{R}^{3}}|\nabla u_\alpha|^2\mathrm{d}x
-\frac{1}{6}\int_{\mathbb{R}^{3}}\tilde{V}u_\alpha^2\mathrm{d}x\\
&=&\frac{(p-2)\beta}{2p}\int_{\mathbb{R}^{3}}|u_\alpha|^p\mathrm{d}x
+\frac{(q-2)}{2q}\int_{\mathbb{R}^{3}}|u_\alpha|^q\mathrm{d}x\notag\\
&\le&\frac{C_q(q-2)}{2q}\bigg(\int_{\mathbb{R}^{3}}|u_\alpha|^2\mathrm{d}x\bigg)^{\frac{6-q}{4}}
\bigg(\int_{\mathbb{R}^{3}}|\nabla u_\alpha|^2\mathrm{d}x\bigg)^{\frac{3(q-2)}{4}}.
\end{eqnarray*}
Thus, there holds for $u_\alpha \neq 0$:

\begin{equation}\label{equ:241224-e1}
    \int_{\mathbb{R}^{3}}|\nabla u_\alpha|^2\mathrm{d}x \geq \left(\frac{q(2a-\|\tilde{V}_+\|_{\frac{3}{2}}S^{-1})}{3C_q(q-2)}\right)^{\frac{4}{3(q-2)-4}} \bigg(\frac{1}{\alpha}\bigg)^{\frac{6-q}{3(q-2)-4}}.
\end{equation}

On the other hands,
\begin{eqnarray*}
\frac{2-q}{2q}\lambda_\alpha\int_{\mathbb{R}^{3}}u_\alpha^2\mathrm{d}x
&=&\frac{q-6}{6q}(a+b\Lambda^2)\int_{\mathbb{R}^{3}}|\nabla u_\alpha|^2\mathrm{d}x
+\frac{1}{6}\int_{\mathbb{R}^{3}}\tilde{V}u_\alpha^2\mathrm{d}x\\
&&\quad+\frac{q-2}{2q}\int_{\mathbb{R}^{3}}Vu_\alpha^2\mathrm{d}x
-\frac{(q-p)\beta}{pq}\int_{\mathbb{R}^{3}}|u_\alpha|^p\mathrm{d}x\notag\\
&\le&\frac{q-6}{6q}(a+b\Lambda^2)\int_{\mathbb{R}^{3}}|\nabla u_\alpha|^2\mathrm{d}x
+\frac{\|\tilde{V}\|_\infty}{6}\alpha+\frac{q-2}{2q}\|V\|_\infty\alpha\\
&&\quad-\frac{(q-p)\beta C_{P}}{pq}\alpha^{\frac{6-p}{4}}\bigg(\int_{\mathbb{R}^{3}}|\nabla u_\alpha|^2\mathrm{d}x\bigg)^{\frac{3(p-2)}{4}}.
\end{eqnarray*}

From the above two inequalities, we conclude that
\begin{eqnarray}\label{equ:241224-e2}\nonumber
\frac{2-q}{2q}\lambda_\alpha\int_{\mathbb{R}^{3}}u_\alpha^2\mathrm{d}x
&\le&\frac{q-6}{6q}(a+b\Lambda^2)\bigg(\frac{q(2a-\|\tilde{V}_+\|_{\frac{3}{2}}S^{-1})}{3C_q(q-2)}\bigg)^{\frac{4}{3q-10}}\alpha^{\frac{q-6}{3q-10}}
+\frac{\|\tilde{V}\|_\infty}{6}\alpha\\  \nonumber
&&\quad+\frac{q-2}{2q}\|V\|_\infty\alpha-\frac{(q-p)\beta C_{P}}{pq}\alpha^{\frac{6-p}{4}}\bigg(\int_{\mathbb{R}^{3}}|\nabla u_\alpha|^2\mathrm{d}x\bigg)^{\frac{3(p-2)}{4}}\\
&\to&-\infty\quad\text{as}\ \alpha\to0.
\end{eqnarray}
Therefore,  $\lambda_\alpha>0$ for $\alpha>0$ small.

\textbf{Case2: }There is a sequence $\alpha_n\to0$
such that $u_{\alpha_n}=0$. In view of \eqref{g-alpha}, $u_{\alpha_n}=0$ for
any $\alpha_n\in\left(0,\alpha'_2\right)$. Let $z_{r,\alpha_n}\in\Omega_r$ be such that $u_{r,\alpha_n}(z_{r,\alpha_n})=\max_{\Omega_r}u_{r,\alpha_n}$, it holds $|z_{r,\alpha_n}|\to\infty$ as $r\to\infty$. Otherwise, there exists $z_0\in\mathbb{R}$ such
that $z_{r,\alpha_n}\to z_0$, and hence $u_{\alpha_n}(z_0)\ge d_{\alpha_n}> 0$. This contradicts
$u_{\alpha_n}=0$.  Moreover, dist$(z_{r,\alpha_n},\partial\Omega_r)\to\infty$ as $r\to\infty$ by an argument similar to that in Lemma \ref{lambda-r>0-betale0}.
Now, for $n$ fixed, let $v_r(x)=u_{r,\alpha_n}(x+z_{r,\alpha_n})$ for any $x\in\Sigma_r:=\{x\in\mathbb{R}^3:x+z_{r,\alpha_n}\in\Omega_r\}$.
It follows from Lemma \ref{nabla-u-bdd} that $v_r$ is bounded in $H^1(\mathbb{R}^3)$, and there is $v\in H^1(\mathbb{R}^3)$ with $v\neq0$ such that $v_r\rightharpoonup v$ as $r\to\infty$.
Observe that for every $\varphi\in C_c^\infty(\Omega_r)$, there is  $r$ large such that $\varphi(\cdot-z_{r,\alpha_n})\in C_c^\infty(\Omega_r)$ due to dist$(z_{r,\alpha_n},\partial\Omega_r)\to\infty$ as $r\to\infty$. It then follows that
\begin{eqnarray}\label{varphi-z}
&&\int_{\Omega_r}|u_{r,\alpha_n}|^{q-2}u_{r,\alpha_n}\varphi(\cdot-z_{r,\alpha_n})\mathrm{d}x
+\beta\int_{\Omega_r}|u_{r,\alpha_n}|^{p-2}u_{r,\alpha_n}\varphi(\cdot-z_{r,\alpha_n})\mathrm{d}x\notag\\
&=&(a+b\int_{\Omega_r}|\nabla u_{r,\alpha_n}|^2\mathrm{d}x)\int_{\Omega_r}\nabla u_{r,\alpha_n}\nabla\varphi(\cdot-z_{r,\alpha_n})\mathrm{d}x
+\int_{\Omega_r}Vu_{r,\alpha_n}\varphi(\cdot-z_{r,\alpha_n})\mathrm{d}x\notag\\
&&\quad
+\lambda_{r,\alpha_n}\int_{\Omega_r}u_{r,\alpha_n}\varphi(\cdot-z_{r,\alpha_n})\mathrm{d}x.
\end{eqnarray}
Since $|z_{r,\alpha_n}|\to\infty$ as $r\to\infty$, we have
\begin{eqnarray}\label{v-0}
\Big|\int_{\Omega_r}Vu_{r,\alpha_n}\varphi(\cdot-z_{r,\alpha_n})\mathrm{d}x\Big|
&\le&\int_{\text{Supp}\varphi}\Big|V(\cdot+z_{r,\alpha_n})v_{r}\varphi\Big|\mathrm{d}x\notag\\
&\le&\|v_{r}\|_{6}\|\varphi\|_{6}
\bigg(\int_{\text{Supp}\varphi}|V(\cdot+z_{r,\alpha_n})|^{\frac{3}{2}}\mathrm{d}x\bigg)^{\frac{2}{3}}\notag\\
&\le&\|v_{r}\|_{6}\|\varphi\|_{6}
\bigg(\int_{\mathbb{R}^{3}\setminus B_{\frac{|z_{r,\alpha_n}|}{2}}}|V(\cdot+z_{r,\alpha_n})|
^{\frac{3}{2}}\mathrm{d}x\bigg)^{\frac{2}{3}}\notag\\
&\to&0\quad\text{as}\ r\to\infty.
\end{eqnarray}
Letting $r\to\infty$ in \eqref{varphi-z} we obtain for $\varphi\in C_c^\infty(\mathbb{R}^{3})$:
\[
(a+b\Lambda^2)\int_{\mathbb{R}^{3}}\nabla v\cdot\nabla\varphi\mathrm{d}x
+\lambda_{\alpha_n}\int_{\mathbb{R}^{3}}v\varphi\mathrm{d}x
=\int_{\mathbb{R}^{3}}|v|^{q-2}v\varphi\mathrm{d}x
+\beta\int_{\mathbb{R}^{3}}|v|^{p-2}v\varphi\mathrm{d}x.
\]
 Now we argue as in the case that $u_\alpha\neq0$ above, there exists $\tilde{\alpha}<\alpha'_2$ such that $\lambda_\alpha>0$ for any $\alpha\in\left(0,\tilde{\alpha}\right)$, the proof is now complete.
\end{proof}

\noindent\textbf{Proof of Theorem \ref{betale0-e>0-Omega}:} The proof is an immediate consequence of Theorem \ref{beta>0-s-omega}, Lemma \ref{lambda-r>0-betale0}
and Remark \ref{ur-ubdd-beta}.\qed

\section{Normalized solution in $\mathbb{R}^{3}$}

The aim of this section is to obtain the existence of normalized solution to Eq.\eqref{main-eq} in  $\mathbb{R}^{3}$. For this purpose, under the assumptions of  Theorem \ref{beta>0-e<0-Omega}(resp. Theorem \ref{beta>0-e>0-Omega} and Theorem \ref{betale0-e>0-Omega}), we first obtain the solutions $(\lambda_r,u_r)$ to \eqref{main-eq-omega} with $\Omega_r=B_r$, see  Theorem \ref{beta>0-e<0-Omega}(resp. Theorem \ref{beta>0-e>0-Omega} and Theorem \ref{betale0-e>0-Omega}). Then by Theorem \ref{th-b-e-r-lambda} (resp. Lemma \ref{beta>0-e-strong-ge} and Lemma \ref{lambda-r>0-betale0}), we conclude that $u_r$ is bounded uniformly in $r$ and $\underset{r\to\infty}{\liminf}\,\lambda_r>0$.
The next lemma is to analyze the compactness of normalized solutions $u_r$ for \eqref{main-eq-s-omega} with $\Omega=B_r$ as  $r\to\infty$.

\begin{lemma}\label{rinftycom}
Under the assumptions of  Theorem \ref{beta>0-e<0-Omega}(resp. Theorem \ref{beta>0-e>0-Omega} and Theorem \ref{betale0-e>0-Omega}), let $\{(\lambda_r,u_r)\}$ be a sequence of solutions of \eqref{main-eq-omega} with $\Omega_r=B_r$. Then there exist a subsequence (still denoted by $\{(\lambda_r,u_r)\}$), a $u_0\in H^1(\mathbb{R}^{3})$ satisfying
\[
  -(a+bB^2)\Delta u+Vu+\lambda u=|u|^{q-2}u+\beta|u|^{p-2}u\quad\text{in}\ \mathbb{R}^{3},
\]
$l\in\mathbb{N}\cup\{0\}$, nontrivial solutions  $w^1,...,w^l\in H^1(\mathbb{R}^3)\setminus\{0\}$ of the following problem
  \begin{equation}\label{omega}
  -(a+bB^2)\Delta u+\lambda u=|u|^{q-2}u+\beta|u|^{p-2}u\quad\text{in}\ \mathbb{R}^{3}
  \end{equation}
such that
\[
\lambda_r\to\lambda>0,\quad\Big\|u_r-u_0-\Sigma_{k=1}^{l}w^k(\cdot-z_r^k)\Big\|_{H^1}\to0\quad \text{as}\ r\to\infty,
\]
where $B:=\lim_{r\to\infty}\int_{\mathbb{R}^{3}}|\nabla u_r|^2\mathrm{d}x$ and $\{z^k_r\}\subset\mathbb{R}^{3}$ with $z^k_r\in B_r$  satisfying
\begin{equation}\label{zk-dist}
|z^k_r|\to\infty,\quad dist(z^k_r,\partial B_r)\to\infty,\quad|z^k_r-z^{k'}_r|\to\infty
\end{equation}
for any $k,k'=1,...,l$ and $k\neq k'$.
\end{lemma}

\begin{proof}
   Recall the fact that $u_r$ is bounded uniformly in $r$, there exist $ u_0\in H^1(\mathbb{R}^{3})$ and $\lambda>0$ such that, up to a subsequence,
\[
\lambda_r\to\lambda,\quad u_r\rightharpoonup u_0\neq 0\quad\text{in}\ H^1(\mathbb{R}^{3}).
\]
Moreover, $u_0$ is a solution of
\[
-(a+bB^2)\Delta u+Vu+\lambda u=|u|^{q-2}u+\beta|u|^{p-2}u\quad\text{in}\ \mathbb{R}^{3}.
\]
Set $\nu_r^1=u_r-u_0$, then $\nu_r^1\rightharpoonup0$ as $r\to\infty$ in $H^1(\mathbb{R}^{3})$.
We divide the proof into three steps:

\textbf{Step1: } The vanishing case occurs. That is,
\[
\underset{r\to\infty}{\limsup}\underset{z\in\mathbb{R}^{3}}{\sup}\int_{B(z,1)}|\nu_r^1|^2\mathrm{d}x=0.
\]
By the Brezis-Lieb lemma, the concentration compactness principle and the Vitali convergence theorem,
\begin{eqnarray*}
o_r(1)&=&\int_{\mathbb{R}^{3}}|\nu_r|^{q}\mathrm{d}x
+\beta\int_{\mathbb{R}^{3}}|\nu_r|^{p}\mathrm{d}x\\
&=&\int_{\mathbb{R}^{3}}|u_r|^{q}\mathrm{d}x
+\beta\int_{\mathbb{R}^{3}}|u_r|^{p}\mathrm{d}x
-\int_{\mathbb{R}^{3}}|u_0|^{q}\mathrm{d}x
-\beta\int_{\mathbb{R}^{3}}|u_0|^{p}\mathrm{d}x+o_r(1)\\
&=&(a+b\int_{\mathbb{R}^{3}}|\nabla u_r|^2\mathrm{d}x)\int_{\mathbb{R}^{3}}|\nabla u_r|^2\mathrm{d}x+\int_{\mathbb{R}^{3}}Vu_r^2\mathrm{d}x
+\lambda\int_{\mathbb{R}^{3}}|u_r|^2\mathrm{d}x\\
&&\quad-(a+bB^2)\int_{\mathbb{R}^{3}}|\nabla u_0|^2\mathrm{d}x-\int_{\mathbb{R}^{3}}Vu_0^2\mathrm{d}x
-\lambda\int_{\mathbb{R}^{3}}|u_0|^2\mathrm{d}x\\
&\ge&a\int_{\mathbb{R}^{3}}|\nabla (u_r-u_0)|^2\mathrm{d}x
+\lambda\int_{\mathbb{R}^{3}}|u_r-u_0|^2\mathrm{d}x+o_r(1).
\end{eqnarray*}
As a consequence $\nu_r^1=u_r-u_0\to 0$ in $H^1(\mathbb{R}^{3})$, and the lemma holds for $l=0$.

\textbf{Step2: }The vanishing does not occur. That is, there exists $z_r^1\in B_r$ such that
\begin{equation}\label{vrd}
\int_{B(z_r^1,1)}|\nu_r^1|^2\mathrm{d}x>d>0.
\end{equation}
We first claim that $|z_r^1|\to\infty$ as $r\to\infty$. Assume  by contradiction that  there is $R>0$ such that $|z_r^1|<R$,
\[
\int_{B_{R+1}}|\nu_r^1|^2\mathrm{d}x\ge\int_{B(z_r^1,1)}|\nu_r^1|^2\mathrm{d}x>d,
\]
contradicting $\nu_r^1\rightharpoonup0$ as $r\to\infty$. Then $|z_r^1|\to\infty$ as $r\to\infty$. Base on this, similar to \eqref{dist} we further have that dist$(z_r^1,\partial B_r)\to\infty$ as $r\to\infty$. Thus $\omega_r^1:=\nu_r^1(\cdot+z_r^1)\rightharpoonup\omega^1$ in $H^1(\mathbb{R}^{N})$, and after a calculation similar to \eqref{v-0}, we get that  $\omega^1$ satisfies
\[
-(a+bB^2)\Delta \omega^1+\lambda\omega^1=|\omega^1|^{q-2}\omega^1+\beta|\omega^1|^{p-2}\omega^1\quad\text{in}\ \mathbb{R}^{3}.
\]
Moreover, by \eqref{vrd},
\[
\int_{B_1}|\omega^1|^2\mathrm{d}x
=\underset{r\to\infty}{\lim}\int_{B_1}|\nu_r^1(x+z_r^1)|^2\mathrm{d}x
=\underset{r\to\infty}{\lim}\int_{B(z_r^1,1)}|\nu_r^1(x)|^2\mathrm{d}x>d>0.
\]
This implies $\omega^1\neq0$. Setting $v_r^2(x)=\nu_r^1(x)-\omega^1(x-z_r^1)$, then $\nu_r^2\rightharpoonup0$ as $r\to\infty$ in $H^1(\mathbb{R}^{3})$.
\textbf{Case1:} If
\[
\underset{r\to\infty}{\limsup}\underset{z\in\mathbb{R}^{3}}{\sup}
\int_{B(z_r^1,1)}|\nu_r^2(x)|^2\mathrm{d}x=0
\]
occurs, then we stop and  go to \textbf{step 1}. Consequently the lemma holds for $l=1$.
\textbf{Case2:} If there exists $z_r^2\in B_r$ such that
\[
\int_{B(z_r^2,1)}|\nu_r^2(x)|^2\mathrm{d}x>d>0,
\]
we first \textbf{claim} that $|z_r^1-z_r^2|\to\infty$ as $r\to\infty$. We suppose by the contrary that there is $R>0$ such that $|z_r^1-z_r^2|<R$ for any large $r$. Then
\begin{eqnarray}
\int_{B(z_r^2,1)}|\nu_r^2(x)|^2\mathrm{d}x&\le&\int_{B(z_r^1,R+1)}|\nu_r^2(x)|^2\mathrm{d}x\\
&=&\int_{B(z_r^1,R+1)}|\nu_r^1(x)-\omega^1(x-z_r^1)|^2\mathrm{d}x\notag\\
&=&\int_{B(0,R+1)}|\omega_r^1(x)-\omega^1(x)|^2\mathrm{d}x\notag\\
&\to&0\quad\text{as}\ r\to\infty,
\end{eqnarray}
which is impossible. Then the claim is valid. Base on this, \eqref{zk-dist} holds through a discussion similar to \eqref{dist}, and there is $\omega^2\neq0$ satisfying \eqref{omega}. Let $\nu_r^3=\nu_r^2(x)-\omega^2(x-z_r^2)$, repeating the above process.

\textbf{Step3:} we try to prove that the process  in \textbf{Step2} can be repeated at most finitely many times. For any solution $\omega\neq0$ of \eqref{omega},
let us  first \textbf{claim} $\|\nabla\omega\|_2^2$ is  bounded from below. Indeed, from \eqref{omega}, we have
\[
(a+b\int_{\mathbb{R}^{3}}|\nabla \omega|^2\mathrm{d}x)\int_{\mathbb{R}^{3}}|\nabla \omega|^2\mathrm{d}x
+\lambda\int_{\mathbb{R}^{3}}|\omega|^2\mathrm{d}x
=\int_{\mathbb{R}^{3}}|\omega|^{q}\mathrm{d}x
+\beta\int_{\mathbb{R}^{3}}|\omega|^{p}\mathrm{d}x.
\]
This together with  the Pohozaev identity
\[
\frac{(a+b\int_{\mathbb{R}^{3}}|\nabla \omega|^2\mathrm{d}x)}{6}\int_{\mathbb{R}^{3}}|\nabla \omega|^2\mathrm{d}x
+\frac{\lambda}{2}\int_{\mathbb{R}^{3}}|\omega|^2\mathrm{d}x
=\frac{1}{q}\int_{\mathbb{R}^{3}}|\omega|^{q}\mathrm{d}x
+\frac{\beta}{p}\int_{\mathbb{R}^{3}}|\omega|^{p}\mathrm{d}x
\]
leads to
\[
\frac{(a+b\int_{\mathbb{R}^{3}}|\nabla \omega|^2\mathrm{d}x)}{3}\int_{\mathbb{R}^{3}}|\nabla \omega|^2\mathrm{d}x
=\frac{(q-2)}{2q}\int_{\mathbb{R}^{3}}|\omega|^{q}\mathrm{d}x
+\frac{\beta(p-2)}{2p}\int_{\mathbb{R}^{3}}|\omega|^{p}\mathrm{d}x.
\]
When the case $\beta\le0$ occurs,  it follows from the Gagliardo-Nirenberg inequality that
\begin{eqnarray*}
\frac{a}{3}\int_{\mathbb{R}^{3}}|\nabla \omega|^2\mathrm{d}x&\le&\frac{(a+b\int_{\mathbb{R}^{3}}|\nabla \omega|^2\mathrm{d}x)}{3}\int_{\mathbb{R}^{3}}|\nabla \omega|^2\mathrm{d}x\\
&\le&\frac{(q-2)}{2q}\int_{\mathbb{R}^{3}}|\omega|^{q}\mathrm{d}x\\
&\le&\frac{(q-2)C_q}{2q}\alpha^{\frac{6-q}{2}}\bigg(\int_{\mathbb{R}^{3}}|\nabla \omega|^2\mathrm{d}x\bigg)^{\frac{3(q-2)}{4}}.
\end{eqnarray*}
Then $\|\nabla\omega\|_2^2$ is  bounded from below due to $\frac{14}{3}<q<6$.
When the case $\beta>0$ occurs. As we know, the ground state energy of   \eqref{omega} is positive due to it is the mountain pass value. That is,
\[
\frac{1}{2}a\int_{\mathbb{R}^{3}}|\nabla \omega|^2\mathrm{d}x
+\frac{1}{4}b\bigg(\int_{\mathbb{R}^{3}}|\nabla \omega|^2\mathrm{d}x\bigg)^2
-\frac{1}{q}\int_{\mathbb{R}^{3}}|\omega|^{q}\mathrm{d}x
-\frac{\beta}{p}\int_{\mathbb{R}^{3}}|\omega|^{p}\mathrm{d}x>0.
\]
This together with  the Pohozaev identity leads to
\[
\frac{(a+b\int_{\mathbb{R}^{3}}|\nabla \omega|^2\mathrm{d}x)}{3}\int_{\mathbb{R}^{3}}|\nabla \omega|^2\mathrm{d}x>0.
\]
Therefore, $\|\nabla\omega\|_2^2$ is  bounded from below in $\beta\in\mathbb{R}$.
We suppose that it can be repeated at most m times. Then
\begin{equation}
\underset{r\to\infty}{\lim}\|u_r\|_{H^1}\ge\|u_0\|_{H^1}
+\Sigma_{k=1}^{m-1}\|\omega^k(\cdot-z_r^k)\|_{H^1}
\end{equation}
If $m=\infty$, this contradicts with the fact that $u_r$ is bounded uniformly in $r$.
\end{proof}

\noindent\textbf{Proofs of Theorems \ref{beta>0-e<0-rn-}:}   This part follows a standard technique, but let us quickly explain for the reader’s convenience.
Let $\{(\lambda_r,u_r)\}$ be a sequence of solutions of \eqref{main-eq-omega} with $\Omega_r=B_r$.
In view of Lemma \ref{rinftycom}, if $l=0$,  $u_r\to u_0$ as $r\to\infty$ in $H^1(\mathbb{R}^3)$, then Theorems \ref{beta>0-e<0-rn-}  hold. Next, we just need to state that the case $l>0$ is impossible  through a discussion similar to \cite{BQZ24}.
In fact, if $l>0$, without loss of generality, we assume that
$|z^1_r|\le\min\{|z_r^k|:k=1,...,l\}$.
Let
\[
\frac{|z_r^k-z_r^1|}{|z_r^1|}\to d_k\in\left[0,\infty\right],
\quad K=\{k:d_k\neq0\ \text{for}\ k=1,...,l\},\]
and set $v^*=\frac{1}{4}\min\{1,\rho,d^*\}$ with $d^*=\underset{k\in K}{\min}\,d_k>0$. Fixed $\delta\in\left(0,v^*\right)$ and consider the annulus
\[
A_r=B\bigg(z_r^1,\frac{3}{2}\delta|z_r^1|\bigg)\setminus B\bigg(z_r^1,\frac{1}{2}\delta|z_r^1|\bigg).
\]
It is easy to verify
 that $\text{dist}(z_r^k,A_r)\ge\frac{1}{4}\delta|z_r^1|$ for any $k=1,...,l$ and large $r$.
 By Lemma \ref{rinftycom}, one has that  for any $2\le s< 2^*$,
\[
\|u_r\|_{L^s(A_r)}=\Big\|u_0+\Sigma_{k=1}^{l}\omega^k(\cdot-z_r^k)\Big\|_{L^s(A_r)}
\le\|u_0\|_{L^s(A_r)}+\Sigma_{k=1}^{l}\|\omega^k(\cdot-z_r^k)\|_{L^s(A_r)}.
\]
Then
\[\|u_0\|_{L^s(A_r)}=\int_{B(z_r^1,\frac{3}{2}\delta|z_r^1|)}|u_0|^s\mathrm{d}x
\le\int_{ B^c_{\delta|z_r^1|/2}}|u_0|^s\mathrm{d}x\to0\]
and
\[
\|\omega^k(\cdot-z_r^k)\|_{L^s(A_r)}
=\int_{A_r}|\omega^k(\cdot-z_r^k)|^s\mathrm{d}x
\le\int_{B^c_{\delta|z_r^1|/4}}|\omega^k|^s\mathrm{d}x\to0.
\]
As a result,  $\|u_r\|_{L^s(A_r)}\to0$ as $r\to\infty$. Observe that  $u_r$ satisfies
\begin{equation}\label{eq-la}
    -a\Delta u_r+Vu_r+\lambda u_r\le |u_r|^{q-2}u_r+\beta|u_r|^{p-2}u_r\quad\text{in}\ \mathbb{R}^{3}
\end{equation}
for large r, where $\lambda:=\frac{1}{2}\underset{r\to\infty}{\liminf}\,\lambda_r>0$ due to Lemma \ref{rinftycom}. Setting
\[
R_m=\overline{B(z_r^1,\frac{3}{2}\delta|z_r^1|-m)}\setminus B(z_r^1,\frac{1}{2}\delta|z_r^1|+m)\quad\text{for}\ n\in\mathbb{N}^+,
\]
by \cite[Theorem 8.17]{GT83}, one has that  there is constant $C>0$ independent of r such that for large r,
\begin{equation}\label{lambda}
\|u_r\|_{L^\infty(R_1)}^{q-2}\le C\|u_r\|_{L^2(A_r)}^{q-2}<\frac{\lambda}{8}\quad
\text{and}\quad
|\beta|\cdot\|u_r\|_{L^\infty(R_1)}^{p-2}\le C\|u_r\|_{L^2(A_r)}^{p-2}<\frac{\lambda}{8}.
\end{equation}
Let us set $\phi_m=\xi_m(|x-z_r^1|)$, where $\xi_m\in C_c^\infty(\mathbb{R},\left[0,1\right])$ is a cut-off function with  $|\xi_m'(t)\le4|$ for any $t\in R$, and
\[
\xi_m(t)=
\begin{cases}
    1&\text{if}\ \frac{1}{2}\delta|z_r^1|+m<t<\frac{3}{2}\delta|z_r^1|-m,\\
    0&\text{if}\ t<\frac{1}{2}\delta|z_r^1|+m-1\ \text{or}\  t>\frac{3}{2}\delta|z_r^1|-m+1.
\end{cases}
\]
Testing \eqref{eq-la} with $\phi_m^2u_r$, we have
\begin{eqnarray}\label{text-right}
   &&8\int_{R_{m-1}\setminus R_{m}}|u_r|\cdot|\nabla u_r|\mathrm{d}x\notag\\
   &\ge&a\int_{R_{m-1}}|\nabla u_r|^2\phi_m^2\mathrm{d}x
    +\int_{R_{m-1}}V| u_r|^2\phi_m^2\mathrm{d}x
    +\lambda\int_{R_{m-1}}| u_r|^2\phi_m^2\mathrm{d}x\notag\\
    &&\quad-\int_{R_{m-1}}| u_r|^{q}\phi_m^2\mathrm{d}x
    -\beta\int_{R_{m-1}}| u_r|^{p}\phi_m^2\mathrm{d}x.
\end{eqnarray}
Since $|z_r^1|\to\infty$ and $\underset{|x|\to\infty}{\liminf}\,V(x)=0$ there exists $\bar{r}$ such that $V(x)\ge-\frac{\lambda}{4}$ for any  $x\in A_r$ with $r>\bar{r}$. This together with \eqref{lambda}, we get that
\begin{eqnarray}\label{text-left}
    &&\min\big\{a,\lambda/2\big\}\bigg(\int_{R_{m-1}}|\nabla u_r|^2\phi_m^2\mathrm{d}x+\int_{R_{m-1}}| u_r|^2\phi_m^2\mathrm{d}x\bigg)\notag\\
    &\le&a\int_{R_{m-1}}|\nabla u_r|^2\phi_m^2\mathrm{d}x+\lambda/2\int_{R_{m-1}}| u_r|^2\phi_m^2\mathrm{d}x\notag\\
    &\le&a\int_{R_{m-1}}|\nabla u_r|^2\phi_m^2\mathrm{d}x
    +\int_{R_{m-1}}V| u_r|^2\phi_m^2\mathrm{d}x
    +\lambda\int_{R_{m-1}}| u_r|^2\phi_m^2\mathrm{d}x\notag\\
    &&\quad-\int_{R_{m-1}}| u_r|^{q}\phi_m^2\mathrm{d}x
    -\beta\int_{R_{m-1}}| u_r|^{p}\phi_m^2\mathrm{d}x.
\end{eqnarray}
It follows from \eqref{text-right}-\eqref{text-left} that
\[\int_{R_{m}}|\nabla u_r|^2\phi_m^2\mathrm{d}x+\int_{R_{m}}| u_r|^2\phi_m^2\mathrm{d}x\le\kappa\bigg(\int_{R_{m}\setminus R_{m-1}}|\nabla u_r|^2\phi_m^2\mathrm{d}x+\int_{R_{m}\setminus R_{m-1}}|u_r|^2\phi_m^2\mathrm{d}x\bigg),\]
where $\kappa=4\max\{a,\frac{2}{\lambda}\}$.
From this, one has that
\[
a_m\le\theta a_{m-1}\le \theta^m\underset{r}{\max}\|u_r\|_{H^1}=\big(\underset{r}{\max}\|u_r\|_{H^1}\big)e^{mln\theta}=:\tilde{A}e^{mln\theta},
\]
where \[
a_m:=\int_{R_{m}}|\nabla u_r|^2\mathrm{d}x+\int_{R_{m}}| u_r|^2\mathrm{d}x\quad\text{and}\quad \theta=\frac{\kappa}{\kappa+1}.
\]
Let
\[
D_r:=B(z_r^1,\delta|z_r^1|+1)\setminus B(z_r^1,\delta|z_r^1|-1)\subset R_M
\]
with
$M=\bigg[\frac{\delta|z_r^1|}{2}\bigg]-1.$
We further have
\[\int_{D_r}|\nabla u_r|^2\mathrm{d}x+\int_{D_r}| u_r|^2\mathrm{d}x
\le \tilde{A}e^{Mln\theta}\le\tilde{A}e^{\frac{\delta|z_r^1|}{2}ln\theta}.\]
It follows from  \cite[Theorem 8.17 ]{GT83} that
\begin{equation}\label{e-u}
|u_r(x)|\le C\|u_r\|_{L^2(D_r)}\le A e^{-c|z_r^1|}
\end{equation}
for any $x\in D_{r,\frac{2}{3}}:=\delta|z_r^1|-\frac{2}{3}<|x-z_r^1|<\delta|z_r^1|+\frac{2}{3}$ and large r; here A, c, C
are independent of r.
From the $L^p$-estimates of elliptic partial differential equations \cite[Theorem 9.11]{GT83},  we further conclude that
\begin{equation}\label{e-nabla-u}
    |\nabla u_r(x)|\le A e^{-c|z_r^1|}
\end{equation}
for any $x$ with $\delta|z_r^1|-\frac{1}{2}<|x-z_r^1|<\delta|z_r^1|+\frac{1}{2}$.

With \eqref{e-u}-\eqref{e-nabla-u} in hand,  let $\Gamma_1=\partial B(z_r^1,\delta|z_r^1|)\cap B_r$ and $\Gamma_2=B(z_r^1,\delta|z_r^1|)\cap\partial B_r$. Multiplying both sides
of \eqref{main-eq-omega} with $z_r^1\cdot\nabla u_r$ and integrating, we obtain
\begin{eqnarray}\label{a-a2}
    \frac{1}{2}\int_{B(z_r^1,\delta|z_r^1|)\cap B_r}(z_r^1\cdot\nabla V)u_r^2\mathrm{d}\sigma&=&\frac{(a+b\int_{\Omega_r}|\nabla u|^2\mathrm{d}x)}{2}\int_{\Gamma_1\cup\Gamma_2}(z_r^1\cdot\textbf{n})|\nabla u_r|^2\mathrm{d}\sigma\notag\\
    &&\quad-(a+b\int_{\Omega_r}|\nabla u|^2\mathrm{d}x)\int_{\Gamma_1\cup\Gamma_2}(\textbf{n}\cdot\nabla u_r)(z_r^1\cdot\nabla u_r)\mathrm{d}\sigma\notag\\
    &&\quad-\int_{\Gamma_1\cup\Gamma_2}(z_r^1\cdot\textbf{n})
    (\frac{Vu_r^2}{2}-\frac{\beta u_r^p}{p}-\frac{u_r^{q}}{q}+\frac{\lambda|u_r|^2}{2})\mathrm{d}\sigma\notag\\
    &=:&\mathcal{A}_1+\mathcal{A}_2,
\end{eqnarray}
where $\textbf{n}$ is the outward unit normal vector on $\partial{B(z_r^1,\delta|z_r^1|)\cap B_r}$,
\begin{eqnarray*}
    \mathcal{A}_i:&=&\frac{(a+b\int_{\Omega_r}|\nabla u|^2\mathrm{d}x)}{2}\int_{\Gamma_i}(z_r^1\cdot\textbf{n})|\nabla u_r|^2\mathrm{d}\sigma\notag\\
    &&\quad-(a+b\int_{\Omega_r}|\nabla u|^2\mathrm{d}x)\int_{\Gamma_i}(\textbf{n}\cdot\nabla u_r)(z_r^1\cdot\nabla u_r)\mathrm{d}\sigma\notag\\
    &&\quad-\frac{1}{2}\int_{\Gamma_i}(z_r^1\cdot\textbf{n})
    (\frac{Vu_r^2}{2}-\frac{\beta u_r^p}{p}-\frac{u_r^{q}}{q}+\frac{\lambda|u_r|^2}{2})\mathrm{d}\sigma
\end{eqnarray*}
for $i=1,2$.
As we know $z_r^1\cdot\textbf{n}(x)>0$ for any $x\in\Gamma_2\neq\emptyset$ by the choice of $\delta$.
From the fact that
\[
u_r(x)>0,\quad \text{for any }x\in B_r(x),\  \text{and } u_r(x)=0\ \text{for any }x\in\partial B_r(x),
\]
we deduce $\nabla u_r(x)=-|\nabla u_r(x)|\textbf{n}(x)=-|\nabla u_r(x)|\frac{x}{|x|}$ for any  $x\in\Gamma_2\subset\partial B_r$, which leads to
\begin{equation}\label{a2-le0}
\mathcal{A}_2=-\frac{(a+b\int_{\Omega_r}|\nabla u|^2\mathrm{d}x)}{2}\int_{\Gamma_2}(z_r^1\cdot\textbf{n})|\nabla u_r|^2\mathrm{d}\sigma\le0.
\end{equation}
In view of \eqref{e-u}-\eqref{e-nabla-u},  there is $\tau_0>0$ independent of r such that
\[
\underset{r\to\infty}{\limsup} e^{\tau_0|z_r^1|}\mathcal{A}_1=0.
\]
But it follows from Lemma \ref{rinftycom} that
\[
\underset{r\to\infty}{\limsup}\int_{B(z_r^1,\delta|z_r^1|)\cap B_r}u_r^2\mathrm{d}x\ge\|\omega^1\|_2^2>0.
\]
Moreover, the definition of $\delta$ and the assumption $(V_1)$ imply
\begin{eqnarray*}\label{e>0}
    \underset{r\to\infty}{\lim}
e^{\tau_0|z_r^1|}\int_{B(z_r^1,\delta|z_r^1|)\cap B_r}(z_r^1\cdot\nabla V)u_r^2\mathrm{d}x \ge \underset{r\to\infty}{\lim}
    e^{\tau_0|z_r^1|}\bigg(\underset{x\in B(z_r^1,\delta|z_r^1|)}{\inf}(z_r^1\cdot\nabla V)\bigg)\int_{B(z_r^1,\delta|z_r^1|)\cap B_r}u_r^2\mathrm{d}x>0.
\end{eqnarray*}
This contradicts \eqref{a-a2}-\eqref{a2-le0}. The proof is now complete.\qed\\

\noindent\textbf{Acknowledgement}\ \ 

\bibliographystyle{amsplain}

\begin{thebibliography}{20}







\bibitem{ABD}Adami,R., Boni,F.,  Dovetta,S., Competing nonlinearities in NLS equations as source of
threshold phenomena on star graphs. J. Funct. Anal., 283(1):Paper No. 109483, 34, (2022).

\bibitem{BV13} Bartsch, T., de Valeriola, S. Normalized solutions of nonlinear Sch\"{o}dinger equations. Arch. Math. 100(1):75--83, (2013).


\bibitem{BQZ24} Bartsch, T., Qi, S.J., Zou, W.M.: Normalized solutions to Schr\"{o}dinger equations with potential and inhomogeneous nonlinearities on large smooth domains. Math. Ann., \textbf{390},  4813--4859 (2024).


\bibitem{BCJ}Borthwick,J., Chang,X., Jeanjean,L., and Soave,N., Normalized solutions of $L^2$-supercritical NLS equations on noncompact metric graphs with localized nonlinearities. Nonlinearity,
36(7):3776--3795, 2023.
\bibitem{BCJS10} Borthwick, J., Chang, X., Jeanjean, L., Soave, N.: Bounded Palais-Smale sequences with Morse type information for some constrained functionals. Transactions of the American Mathematical Society, \textbf{377}(6), 4481--4517, (2024).

\bibitem{CZ23}
 Cai, L.; Zhang, F.B. Normalized solutions of mass supercritical Kirchhoff equation with potential. J. Geom. Anal. \textbf{33}(3), Paper No. 107, (2023).


		




\bibitem{CL91} Chen, W., Li, C.: Classification of solutions of some nonlinear elliptic equations, Duke Math. J., \textbf{63}(3),  615--622 (1991).


 \bibitem{CT2024}	
Chen, S.T., Tang, X.H.: Normalized solutions for Kirchhoff equations with Sobolev critical exponent and mixed nonlinearities.  Math. Ann., \textbf{390},  4813--4859 (2024).

\bibitem{DZ22}
 Ding, Y.H., Zhong, X.X.:   Normalized solution to the Sch\"{o}dinger equation with potential and general nonlinear term: Mass super-critical case. J. Differ. Equ., {\bf334}, 194--215 (2022).

\bibitem{DPS}  Deng, Y.B.,  Peng, S.J., Shuai, W.: Existence and asymptotic behavior of nodal solutions for the
Kirchhoff-type problems in $\R^3$. J. Funct. Anal., {\bf269}, 3500--3527 (2015).

\bibitem{EL82} Esteban, M., Lions, P.L.: Existence and nonexistence results for semilinear elliptic problems in unbounded domains. Proc. R. Soc. Edinb. Sect. A, \textbf{93}, 1--14 (1982).




\bibitem{F} Figueiredo, G.M.,  Ikoma,N., Santos Junior, J.R.: Existence and concentration result for the Kirchhoff type equations with general nonlinearities, Arch. Ration. Mech. Anal., {\bf 213}  931--979 (2014).
\bibitem{FLZ2023}
Feng, X.J., Liu, H.D., Zhang, Z.T.: Normalized solutions for Kirchhoff type equations with combined nonlinearities: the Sobolev critical case. Discrete Contin. Dyn. Syst., \textbf{43}(8), 2935--2972 (2023).

\bibitem{GJ} Gou, T., Jeanjean, L. Multiple positive normalized solutions for nonlinear Schr\"{o}dinger systems. Nonlinearity. 31(5):2319--2346, (2018).

\bibitem{GT83} Gilbarg, D., Trudinger, N.: Elliptic Partial Differential Equations of Second Order. Springer, Berlin (1983).

\bibitem{G} Guo, Z.J.: Ground states for Kirchhoff equations without compact condition, J. Differ. Equ.,
{\bf259}, 2884--2902 (2015).

\bibitem{HLZZ}  He, Q.H.,  Lv, Z.Y.,  Zhang, Y.M.,  Zhong, X.X.: Positive normalized solution to the Kirchhoff equation with general nonlinearities of mass super-critical. J. Differ. Equ., \textbf{356}, 375--406 (2023).

\bibitem{HLT}
He, Q.H.,  Lv, Z.Y., Tang, Z.W.: The existence of normalized solutions to the Kirchhoff equation with potential and Sobolev critical nonlinearities. J. Geom. Anal., \textbf{33}(7), Paper No. 236, (2023).

\bibitem{H}  He, Y.: Concentrating bounded states for a class of singularly perturbed Kirchhoff type equations with a general nonlinearity. J. Differ. Equ., {\bf261}(11), 6178--6220 (2016).

\bibitem{HSbook}
 Hislop, P.D.,  Sigal, I.M.: Introduction to Spectral Theory, in: Applied Mathematical Sciences, vol. 113, Springer-Verlag,
New York, With applications to Schr\"{o}dinger operators.

\bibitem{JeanZhangZhong2021}
Jeanjean, L., Zhang, J.J.,  Zhong, X.X.:
A global branch approach to normalized solutions for the Schr\"{o}dinger equation. J. Math. Pures Appl., \textbf{183}(9), 44--75 (2024).




\bibitem{Ki}  Kirchhoff, G.: Mechanik, Teubner, Leipzig, (1883).

\bibitem{Kwong1989} Kwong, M.K.: Uniqueness of positive solutions of $-\Delta u-u+u^p=0$ in $\mathbb{R}^{N}$, Arch. Ration. Mech. Anal., \textbf{105}(3), 243--266 (1989).





\bibitem{JLLions1978}
 Lions, J.L.:
\newblock On some questions in boundary value problems of mathematical physics, in: Contemporary Developments in Continuum Mechanics and Partial Differential Equations, Proceedings of International Symposium, Inst. Mat., Univ. Fed. Rio de Janeino, 1977, in: North-Holl. Math. Stud., vol. 30, North-Hollad, Amsterdam, 284--346 (1978).

\bibitem{LY}   Li, G.B., Ye, H.Y.: Existence of positive ground state solutions for the nonlinear Kirchhoff type equations in $\R^3$, J. Differ. Equ., {\bf257}  566--600 (2014).
\bibitem{LLY} Li, G.B., Luo, X., Yang, T.: Normalized solutions to a class of kirchhoff equations with sobolev critical expoent,
Ann. Fenn. Math., {\bf47}, 895--925, (2022).

\bibitem{13} Li, G.B., Luo, P., Peng, S.J.,  Wang, C.H., Xiang.C.-L.:  A singularly perturbed Kirchhoff problem revisited. J. Differ. Equ., {\bf268}(2), 541--589, (2020).

\bibitem{14}  Luo, P.,  Peng, S.J.,  Wang, C.H.,  Xiang, C.-L.:  Multi-peak positive solutions to a class of Kirchhoff equations. Proc. Royal. Soc. Edinburgh. A, {\bf149}(4), 1097--1122 (2019).

\bibitem{LZ24} Lin, X.L., Zheng, S.Z.: On multiplicity and concentration for a magnetic Kirchhoff-Schr\"{o}dinger equation involving critical exponents in $\mathbb{R}^2$. Z. Angew. Math. Phys., \textbf{75}, 112 (2024).  https://doi.org/10.1007/s00033-024-02260-5




\bibitem{BHG3}Noris, B., Tavares, H., Verzini, G.:  Normalized solutions for nonlinear Schr\"odinger systems on bounded domains. Nonlinearity, \textbf{32}(3), 1044--1072, (2019).
\bibitem{PPV} Pellacci, B., Pistoia, A., Vaira, G., Verzini, G. Normalized concentrating solutions to nonlinear elliptic problems. J. Diff. Equ., 275: 882--919, (2021).


\bibitem{RL23}
 Rong, T., Li, F.Y.: Normalized solutions to the mass supercritical Kirchhoff-type equation with non-trapping potential. J. Math. Phys., 64(\textbf{8}), Paper No. 081501,(2023)







\bibitem{Wil96} Willem, M.: Minimax Theorems. Birkhauser, Boston (1996).


\bibitem{Wei82} Weinstein, M.I.: Nonlinear Schr\"{o}dinger equations and sharp interpolation estimates. Commun. Math. Phys., \textbf{87}(4), 567--576 (1982/83).

\bibitem{WC2024}
  Wang, Q., Chang, X.J.: Normalized solutions of $L^2$-supercritical Kirchhoff equations in bounded domains. J. Geom. Anal., \textbf{34}(12), Paper No. 358, (2024).

\bibitem{WT} Wang, J.,  Tian, L.,  Xu, J., Zhang, F.: Multiplicity and concentration of positive solutions for a Kirchhoff type problem with critical growth, J. Differ. Equ., \textbf{253}, 2314--2351 (2012).

\bibitem{Y1} Ye, H.Y.: The sharp existence of constrained minimizers for a class of nonlinear Kirchhoff equations, Math. Methods
Appl. Sci., {\bf38}, 2663--2679 (2015).

\bibitem{Y2}Ye, H.Y.: The existence of normalized solutions for $L^2$-critical constrained problems related to Kirchhoff equations, Z.Angew. Math. Phys., {\bf66}, 1483--1497 (2015).

\bibitem{ZZZZ2021} Zeng, X.Y.,  Zhang, J.J.,  Zhang, Y.M.,  Zhong, X.X.:{\it Positive normalized solution to the Kirchhoff equation with general nonlinearities}. Discrete Contin. Dyn. Syst. Ser. S, \textbf{16}(11) 3394--3409 (2023).
\end{thebibliography}

\end{document}